\documentclass[11pt]{article}

\usepackage[utf8]{inputenc}
\usepackage[T1]{fontenc}
\usepackage{scalefnt}
\usepackage{natbib}
\usepackage{amsfonts}
\usepackage{mathtools}
\usepackage{amssymb}
\usepackage{amsmath}
\usepackage{amsthm}
\usepackage{xcolor}
\definecolor{mydarkblue}{rgb}{0,0.08,0.85}
\usepackage[colorlinks=true,bookmarks=true,linkcolor=mydarkblue,urlcolor=mydarkblue,citecolor=mydarkblue,breaklinks=true]{hyperref}
\usepackage{breakcites}
\usepackage{enumitem}
\usepackage{cases}
\usepackage{commath}
\usepackage{algorithm}
\usepackage{algpseudocode}
\usepackage{interval}
\intervalconfig{soft open fences}
\usepackage{xfrac}
\usepackage{tikz}
\usepackage{tikz-cd}
\usetikzlibrary{matrix}
\usepackage{makecell}
\usepackage{comment}

\usepackage{sectsty}
\makeatletter\def\@seccntformat#1{\protect\makebox[0pt][r]{\csname the#1\endcsname\hspace{12pt}}}\makeatother

\usepackage[verbose=true,letterpaper]{geometry}
\AtBeginDocument{
  \newgeometry{
    textheight=9in,
    textwidth=6in,
    top=1in,
    headheight=12pt,
    headsep=25pt,
    footskip=30pt
  }
}

\usepackage{booktabs}
\usepackage{multirow}
\usepackage{utfsym}

\usepackage[most]{tcolorbox}
\tcbset{
  assumptionstyle/.style={
    enhanced,
    colback=blue!5!white,
    colframe=blue!40!black,
    boxrule=0.4pt,
    arc=2pt,
    outer arc=2pt,
    left=6pt,
    right=6pt,
    top=4pt,
    bottom=4pt,
    fonttitle=\bfseries,
    coltitle=black,
    before skip=0.8\baselineskip,
    after skip=0.8\baselineskip
  }
}
\newtcolorbox{emphasisbox}[1][]{assumptionstyle,title={#1}}

\newtheorem{theorem}{Theorem}[section]
\newtheorem{proposition}[theorem]{Proposition}
\newtheorem{definition}[theorem]{Definition}
\newtheorem{lemma}[theorem]{Lemma}
\newtheorem{corollary}[theorem]{Corollary}
\newtheorem{remark}[theorem]{Remark}
\newtheorem{example}[theorem]{Example}

\newcommand{\reals}{{\mathbb R}}

\newcommand{\grad}{\nabla}

\newcommand{\Exp}{\mathrm{Exp}}
\newcommand{\Log}{\mathrm{Log}}
\newcommand{\inj}{\mathrm{inj}}

\DeclareMathOperator{\dist}{dist}

\DeclareMathOperator{\Proj}{Proj}
\DeclareMathOperator{\codim}{codim}

\DeclarePairedDelimiterX{\inner}[2]{\langle}{\rangle}{#1, #2}
\DeclarePairedDelimiterX{\innerBig}[2]{\Big\langle}{\Big\rangle}{#1, #2}
\DeclarePairedDelimiterX{\innerbig}[2]{\big\langle}{\big\rangle}{#1, #2}
\DeclarePairedDelimiterX{\sinner}[3]{#3\langle}{#3\rangle}{#1, #2}

\newcommand{\D}{\mathrm{D}}

\newcommand{\TODO}[1]{{\color{red}{[#1]}}}

\newcommand\restrold[2]{{
    \left.\kern-\nulldelimiterspace
      #1
    \right|_{#2}
  }}

\makeatletter
\newcommand{\normalscaling}{\bBigg@{0}}

\newcommand{\abitbig}{\bBigg@{1}}
\makeatother

\newcommand\restr[3]{{
    #3.\kern-\nulldelimiterspace
    #1
    #3|_{#2}
  }}

\newcommand{\aref}[1]{\hyperref[#1]{A\ref{#1}}}
 
\newcommand{\cref}[1]{\hyperref[#1]{C\ref{#1}}}
 
\newcommand{\pl}{\ensuremath{\text{P\L}}}

\newcommand{\calE}{\mathcal{E}}
\newcommand{\calF}{\mathcal{F}}

\newcommand{\calM}{\mathcal{M}}
\newcommand{\calN}{\mathcal{N}}

\newcommand{\Hess}{\nabla^2}
\newcommand{\Rd}{\reals^d}
\newcommand{\Rn}{\reals^n}
\newcommand{\Rk}{\reals^k}
\newcommand{\T}{\mathrm{T}}
\newcommand{\N}{\mathrm{N}}

\usepackage{scalerel,stackengine}

\newcommand{\im}{\operatorname{im}}
\newcommand{\ds}{\mathrm{d}s}
\newcommand{\dt}{\mathrm{d}t}
\newcommand{\dtau}{\mathrm{d}\tau}
\newcommand{\ddt}{\frac{\mathrm{d}}{\dt}}
\newcommand{\dds}{\frac{\mathrm{d}}{\ds}}
\newcommand{\Ddt}{\frac{\mathrm{D}}{\dt}}

\newcommand{\Rnn}{\reals^{n \times n}}

\newcommand{\Id}{\operatorname{Id}}

\newcommand{\Sone}{\mathbb{S}^1}
\newcommand{\Stwo}{\mathbb{S}^2}

\title{Smooth, globally Polyak--{\L}ojasiewicz functions are\\nonlinear least-squares}

\author{
Nicolas Boumal\thanks{Institute of Mathematics, EPFL, Lausanne, Switzerland. 
\texttt{\{nicolas.boumal,quentin.rebjock\}@epfl.ch}}
\and
Christopher Criscitiello\thanks{The Wharton School, University of Pennsylvania, USA. 
\texttt{crisciti@wharton.upenn.edu}}
\and
Quentin Rebjock\footnotemark[1]
}

\date{Compiled on \today}

\begin{document}

\maketitle

\begingroup
\renewcommand\thefootnote{}
\footnotetext{Authors listed alphabetically.}
\endgroup

\begin{abstract}
  The Polyak--{\L}ojasiewicz ({P\L}) condition is often invoked in nonconvex optimization because it allows fast convergence of algorithms beyond strong convexity.
  A function $f \colon \mathcal{M} \to \mathbb{R}$ on a Riemannian manifold $\mathcal{M}$ is globally {P\L} if $\|\nabla f(x)\|^2 \geq 2\mu(f(x) - f^*)$ for all $x$, where $f^* = \inf f$ and $\mu > 0$. How much does this pointwise, first-order inequality constrain $f$ and its set of minimizers $S$?

  We show that if $f$ is also smooth ($C^\infty$) and $\mathcal{M}$ is contractible (e.g., if $\mathcal{M} = \mathbb{R}^n$), then the {P\L} condition imposes a firm global structure:
  such a function is necessarily of the form $f(x) = f^* + \|\varphi(x)\|^2$ (a nonlinear sum of squares) where $\varphi \colon \mathcal{M} \to \mathbb{R}^k$ is a submersion, and $k$ is the codimension of $S$ in $\mathcal{M}$.
  The proof hinges on showing that the end-point map of negative gradient flow on $f$ is a trivial smooth fiber bundle over $S$.

  This rigidity leads to a striking dichotomy.
  Either $S$ is diffeomorphic to a Euclidean space, in which case $f$ can be transformed into a convex quadratic by a smooth change of coordinates.
  Or $S$ must display genuinely exotic geometry; for example, it can be diffeomorphic to the Whitehead manifold.

  As a further consequence, we show that there exists a complete Riemannian metric on $\calM$ under which $f$ remains {P\L} and becomes geodesically convex.
\end{abstract}


\paragraph{Keywords:} gradient dominated functions; Morse--Bott; Kurdyka--{\L}ojasiewicz inequality.

%


\section{Introduction} \label{sec:intro}

This paper is about real-valued functions on a Riemannian manifold $\calM$.
In many cases of interest,\footnote{Find a blog post (companion to this article) focused on $\calM = \Rn$ at \href{https://www.racetothebottom.xyz/posts/globalPL}{racetothebottom.xyz/posts/globalPL}.} $\calM$ is simply the Euclidean space of dimension $n$, which we denote by $\Rn$.
Our contributions are new for that case too.
Beyond $\Rn$, we use the following conventions.

\begin{emphasisbox}
  By default, $\calM$ is a smooth Riemannian manifold of finite dimension $n$ that is non-empty, connected and complete.
  Manifolds (and submanifolds) are understood to be without boundary, and \emph{smooth} means $C^\infty$.
\end{emphasisbox}

A function $f \colon \calM \to \reals$ is said to be \emph{globally Polyak--{\L}ojasiewicz} with parameter $\mu > 0$ if it is differentiable and it satisfies the inequality
\begin{align}
  f(x) - f^* \leq \frac{1}{2\mu} \|\grad f(x)\|^2
  \tag{P{\L}}
  \label{eq:PL}
\end{align}
for all $x \in \calM$, where $f^* := \inf_{x \in \calM} f(x)$, $\grad f$ is the gradient of $f$, and $\|\cdot\|$ is the norm on the appropriate tangent space.
If so, we say $f$ is \emph{globally $\mu$-\pl{}} or \emph{globally \pl{}}.

This class includes strongly convex functions as well as many nonconvex ones (see below).
They are of significant interest across various areas of mathematics, and accordingly have been extensively studied.\footnote{Close to a thousand papers mentioning ``Polyak--{\L}ojasiewicz'' are listed on \href{https://scholar.google.com/scholar?hl=en&as_sdt=0\%2C5&as_ylo=2025&as_yhi=2025&q=\%22polyak+lojasiewicz\%22&btnG=}{Google Scholar} for 2025 alone.}
For example, \pl{} functions are often considered in optimization, in part because they allow for non-isolated minimizers, while enabling appreciable convergence guarantees for various algorithms beyond the strongly convex case~\citep{polyak1963gradientmethods,nesterov2006cubic,karimi2016linearPL}.
Just as importantly, they occur in several applications, including control~\citep{fazelLQR2018} and statistics~\citep{chewi20a}.
We refer the reader to Section~\ref{sec:relatedwork} for more context.

As we work to understand what a globally \pl{} function $f$ may look like, it is instructive to consider its \emph{set of critical points $S$}.
It is clear from~\eqref{eq:PL} that this coincides with the set of global minimizers of $f$:
\begin{align}
    S = \big\{ x \in \calM : \grad f(x) = 0 \big\} = \big\{ x \in \calM : x \textrm{ is a global minimizer of } f \big\}.
    \label{eq:S}
\end{align}
This set is non-empty, that is, $f$ attains the value $f^*$ (see the classical Lemma~\ref{lemma:bound-trajectories-QG} below).

\subsection{A not-so-particular case: nonlinear least-squares}
\label{sec:particularcaselstsq}

Beyond strongly convex functions, standard examples consist in functions of the form
\begin{align*}
  f(x) = \frac{1}{2} \|F(x) - b\|^2 && \textrm{ with } && F \colon \calM \to \Rk.
\end{align*}
Minimizing such a function $f$ is called a \emph{nonlinear least-squares} problem (a staple of applied mathematics).
The gradient of $f$ is $\nabla f(x) = \D F(x)^*[F(x) - b]$, where $\D F(x)^*$ denotes the adjoint of the differential of $F$ at $x$.
Let $\sigma(x)$ denote the $k$th singular value of $\D F(x)^*$, so that $\|\nabla f(x)\| \geq \sigma(x) \sqrt{2f(x)}$.
Since $f^* := \inf_x f(x) \geq 0$, this can be restated as
\begin{align*}
  f(x) - f^* \leq f(x) \leq \frac{1}{2\sigma(x)^2} \|\nabla f(x)\|^2.
\end{align*}
Therefore, if (but not only if) $\sigma := \inf_x \sigma(x)$ is positive, we see that $f$ is globally \pl{}.
One of our contributions is to show that, in fact, \emph{all} smooth globally \pl{} functions on $\Rn$ (in particular) are of that form.

\subsection{Characterization of smooth globally \pl{} functions}

We focus on smooth ($C^\infty$) globally \pl{} functions.
Our strongest findings hold when $\calM$ is contractible (which includes $\calM = \Rn$, see Definition~\ref{def:contractible}).
In that setting, we show:
\begin{itemize}
  \item \emph{All} smooth, globally \pl{} functions on $\calM$ are of the special form $f(x) = f^* + \|\varphi(x)\|^2$ where $\varphi \colon \calM \to \Rk$ is a smooth submersion for some $k \geq 0$.
  See Theorem~\ref{thm:plnonlinlstsq}.
  Thus, $f$ is necessarily a \textbf{nonlinear least-squares} function, as in Section~\ref{sec:particularcaselstsq}.
  \item The possible \textbf{sets of minimizers} are clearly characterized:
  \begin{itemize}
    \item For such a function $f$, the set of minimizers $S$ is a contractible smooth manifold (properly embedded in $\calM$)---this follows from both classical and recent results.
    \item The other way around, if $\tilde S$ is \emph{any} contractible smooth manifold, then for each $n > \dim(\tilde S)$ there exists a smooth, globally \pl{} function $f$ on $\Rn$ whose set of minimizers is diffeomorphic to $\tilde S$.
    See Corollary~\ref{cor:equivonS}.
    \item In particular, this means $S$ can be diffeomorphic to a Euclidean space, but it can also be diffeomorphic to an exotic $\reals^4$ or the Whitehead manifold (among others).
    \item If (and only if) $S$ is diffeomorphic to a Euclidean space, there exists a diffeomorphism $\xi \colon \calM \to \Rn$ such that $f(\xi^{-1}(y)) = f^* + y_{m+1}^2 + \cdots + y_n^2$ (with $m = \dim S$).
    See Corollary~\ref{cor:deformintoquadform}.
  \end{itemize}
  \item Such a function $f$ has \textbf{hidden convexity}, in the sense that there exists a complete Riemannian metric on $\calM$ such that $f$ remains globally \pl{} and it becomes geodesically convex.
  See Corollary~\ref{thm:pl-gconvex}.
\end{itemize}

To prove these results and a few more, we consider the map $\pi \colon \calM \to S$ which maps a given point $x$ to the end-point of negative gradient flow on $f$ initialized at $x$, and we show constructively that it defines a trivial smooth fiber bundle with additional control.
We expand on this next.

\subsection{The fiber bundle structure}

As a first step, we prove that \pl{} functions with a \emph{single} minimizer are nonlinear least-squares of a special kind.
This is a corollary to the more general Theorem~\ref{thm:globalsquaredistsinglecp} we prove in Section~\ref{sec:singleminimizer}.
One can think of it as a (presumably folklore) globalized Morse lemma.

The proof uses common techniques from differential topology.
It relies on the standard (local) Morse lemma, the Palais--Cerf theorem and negative gradient flow on $f$.
A point of attention in the proof is to ensure $\varphi$ is a \emph{global} diffeomorphism, including at $x^*$.

\begin{theorem}[easy case] \label{thm:globalsquaredistsinglecppl}
    Let $f \colon \calM \to \reals$ be smooth and globally~\eqref{eq:PL}.
    If $f$ has a unique critical point $x^*$, then there exists a diffeomorphism $\varphi \colon \calM \to \Rn$ such that $f(x) = f(x^*) + \|\varphi(x)\|^2$.
\end{theorem}

In particular, the existence of such a function $f$ implies that $\calM$ is diffeomorphic to Euclidean space.
That part could also be deduced (with some work) from more general results in differential topology such as~\citep[Lem.~3]{milnor1964difftopochapter} and also~\citep[Ex.~15 in \S1.2, p.~21]{hirsch1976differential} (which references \citep{brown1961monotoneunion} for the topological case).
Here, purposefully, we provide an explicit construction for a specific diffeomorphism that reveals the quadratic nature of $f$.

The newer part comes in Section~\ref{sec:generalcase}.
There, we allow $f$ to have more than one minimizer, that is, $S$~\eqref{eq:S} may not be a singleton.
We first show that $S$ necessarily is a submanifold of $\calM$.
Moreover, we show that if $\calM$ is contractible (e.g., $\calM = \Rn$), then $f$ is still a nonlinear least-squares of a special kind.
(Recall $\calM$ is complete and connected.)

\begin{theorem}[general case] \label{thm:plnonlinlstsq}
  Let $f \colon \calM \to \reals$ be smooth and globally~\eqref{eq:PL}.
  Its set $S$ of critical points is a connected smooth manifold, properly embedded in $\calM$.
  Let $k = n - \dim(S)$ be the codimension of $S$.

  Assume $\calM$ is contractible.
  Then there exists a diffeomorphism $\psi \colon \calM \to S \times \Rk$ of the form $\psi = (\pi, \varphi)$ such that $f(x) = f^* + \|\varphi(x)\|^2$, where $f^* = \inf f$.
  Moreover, if $f$ is not constant then $\calM$ is diffeomorphic to $\Rn$.
\end{theorem}

Here too, a point of attention in the proof is to ensure $\psi$ is a \emph{global} diffeomorphism, including on $S$.
This is one part of why we were not able to obtain that result from existing literature---see related work below.

We detail implications (and partial converses) of this theorem in Section~\ref{sec:implications}.
Readily, we can observe that the only contractible complete Riemannian manifolds that admit a (not constant) smooth, globally \pl{} function are those that are diffeomorphic (though not necessarily isometric) to Euclidean space.

Before going to proof techniques, let us comment on the assumptions of Theorem~\ref{thm:plnonlinlstsq}:
\begin{itemize}
  \item We assume throughout that $f$ is $C^\infty$ smooth. 
  With some technical effort, this could be relaxed to $C^p$ with sufficiently large $p$.
  Note that $C^1$ regularity is insufficient, since the minimizer set $S$ may then fail to be a manifold.\footnote{For example, the function $f(x,y)=\frac{x^2y^2}{x^2+y^2}$ is $C^1$ and globally $\pl{}$ on $\mathbb{R}^2$, but it is not $C^2$ and its set of minimizers is a cross, which is not a manifold~\citep[Rem.~2.19]{rebjock2023nonisolated}.} 
  
  \item Likewise, the global \pl{} assumption could be relaxed to cater more precisely to the properties we use in the proof. That said, we should note that \emph{invexity} (that is, the property $\nabla f(x) = 0 \implies f(x) = f^*$) is \emph{not} enough.\footnote{For instance, $f(x, y)=(x^2y-x-1)^2+(x^2-1)^2$ is not \pl{} but it is smooth and invex. Its set of minimizers $S = \{(-1,0),(1,2)\}$ is disconnected, which is incompatible with the conclusions of Theorem~\ref{thm:plnonlinlstsq}.}
  
  \item Finally, \emph{some} condition on $\calM$ is indeed necessary to enable the final conclusions of the theorem. See Section~\ref{sec:commentcontractible} for indications that contractibility is a reasonable choice.
\end{itemize}

\noindent
We discuss these and more in the conclusions and perspective too (Section~\ref{sec:conclusions}).

\subsection{Proof technique} \label{sec:prooftechnique}

To prove Theorem~\ref{thm:plnonlinlstsq}, we begin our study of the \emph{end-point map} $\pi \colon \calM \to S$ in Section~\ref{sec:endpointmap}.
It maps each point $x \in \calM$ to $\pi(x)$, defined as the limit (as time goes to infinity) of negative gradient flow down $f$ when initialized at $x$.
In particular, we show that $\pi$ is continuous in order to argue that $\calM$ \emph{strongly deformation retracts} to $S$ (Definition~\ref{def:deformationretract}, Proposition~\ref{prop:picontinuous}).
This notably implies that $S$ is contractible if and only if $\calM$ is so.
(The construction of deformation retractions via gradient flows is classical~\citep{lojasiewicz1963propriete,kurdyka1998gradientsomin}.)

Using the fundamental theorem of flows together with recent results about the regularity of $S$ (see Lemma~\ref{lem:MB}) and a crucial theorem by \citet{falconer1983limitmapping} (which itself relies on the center stable manifold theorem, see~\citet{hirsch1977invariant}), we show that $S$ is a smooth manifold and that $\pi$ is a smooth submersion (Proposition~\ref{prop:Ssmoothpisubmersion}) with fibers (that is, pre-images $\pi^{-1}(x)$ of individual points $x \in S$) diffeomorphic to $\Rk$ (Proposition~\ref{prop:plrestrictedtofiber}).

This is enough to deduce that $\pi$ is a smooth fiber bundle (Definition~\ref{def:smoothfiberbundle}) owing to a result by~\citet[Cor.~31]{meigniez2002submersions}: see Corollary~\ref{cor:fiberbundlegeneral}.
From here, one might remember that a fiber bundle is \emph{trivial} if its base space is contractible~\citep[\S3.4B]{abraham1988manifolds}.

This is what prompts us to assume $S$ is contractible starting in Section~\ref{sec:fiberbundle}.
Under that assumption, we show that $\pi \colon \calM \to S$ is a \emph{trivial smooth fiber bundle}.
Doing so via the general results just stated would not retain control over the value of $f$.
Instead, we craft explicit trivializations of $\pi$ which are compatible with $f$ in a fruitful way (Theorem~\ref{thm:pitrivial}). 
Additionally, the trivialization is global if $S$ is contractible.
The construction is transparent, and does not require the aforementioned results.

To conclude, we use the fact that $f$ restricted to any fiber of $\pi$ is itself globally \pl{}, though with a unique minimizer (Proposition~\ref{prop:plrestrictedtofiber}).
This allows to conclude with the help of Theorem~\ref{thm:globalsquaredistsinglecppl}.
That last step for the proof of Theorem~\ref{thm:plnonlinlstsq} is given in Section~\ref{sec:proofplnonlinlstsq}, where we prove the more general Theorem~\ref{thm:plnonlinlstsqnocontractibility} that also covers the \emph{local} fiber bundle structure if $S$ is not contractible.

If $\calM$ is contractible and $f$ is not constant, then Theorem~\ref{thm:plnonlinlstsq} shows in particular that $\calM$ is diffeomorphic to $S \times \Rk$ with $k \geq 1$, and that $S$ itself is a contractible smooth manifold.
Under those conditions, $S \times \Rk$ (and hence $\calM$ itself) is diffeomorphic to $\Rn$---this follows from a deep theorem that results from a long line of work by many mathematicians:

\begin{theorem}[\citep{McMillan1961,stallings1962piecewise,luft1987contractible,Perelman20022003}] \label{thm:deepthm2}
  Let $\tilde S$ be a non-empty smooth manifold and fix $k \geq 1$.
  Then, $\tilde S \times \reals^k$ is diffeomorphic to $\reals^{\dim(\tilde S) + k}$ if and only if $\tilde S$ is contractible.
\end{theorem}

\noindent This classical theorem can be stated as: ``stabilizing a contractible smooth manifold by $\reals^k$ produces a Euclidean space.''
Appendix~\ref{app:deeptheorems} outlines how this comes as a consequence of the works cited above, and also \citep{Glimm1960,McMillanZeeman1962,Luft1967}.

\subsection{Implications and converses}\label{sec:implications}

We now examine several implications of Theorem~\ref{thm:plnonlinlstsq}, and some converses.

Throughout, $f \colon \calM \to \reals$ is globally \pl{} and smooth.
Its set of minimizers is $S$~\eqref{eq:S}, with dimension $m = \dim S$ and codimension $k = n - \dim S$.
We further assume $\calM$ is contractible. 
Table~\ref{tab:notation} summarizes recurring notation.

In Section~\ref{sec:Slookslike}, we provide a precise characterization of which manifolds $S$ can arise as minimizing sets of $f$.
Building on this, in Section~\ref{sec:SeqRk} we show that, surprisingly to us, $S$ need not be diffeomorphic to Euclidean space, and we point to sufficient conditions for that to happen anyway.
Finally, in Section~\ref{sec:gconvexintro}, we claim that $f$ is geodesically convex with respect to some complete Riemannian metric on $\calM$.
Some of the proofs are deferred to later sections or appendices.

\begin{table}[t]
\centering
\begin{tabular}{ll}
\toprule
Symbol & Meaning \\
\midrule
$\calM$ & smooth connected complete Riemannian manifold \\
$n$ & dimension of $\calM$ \\
$f \colon \calM \to \mathbb{R}$ & smooth globally Polyak--\L{}ojasiewicz function \\
$\mu$ & global P\L{} parameter \\
$f^*$ & global minimum value $\inf_{x\in \calM} f(x)$ \\
$S$ & set of global minimizers (critical set) of $f$ \\
$m$ & dimension of $S$ \\
$k$ & codimension of $S$ in $\calM$ ($k=n-m$) \\
$\nabla f$ & Riemannian gradient of $f$ \\
$\|\cdot\|$ & norm induced by the Riemannian metric \\
$\mathrm{dist}(\cdot,\cdot)$ & Riemannian distance \\
$\pi \colon \calM \to S$ & end-point map of negative gradient flow on $f$ \\
\bottomrule
\end{tabular}
\label{tab:notation}
\end{table}

\subsubsection{What can $S$ look like?}\label{sec:Slookslike}

Which manifolds can arise as minimizing sets of \pl{} functions on a contractible domain?
Theorem~\ref{thm:plnonlinlstsq} already tells us quite a lot: $S$ must be a contractible smooth manifold.\footnote{The full strength of Theorem~\ref{thm:plnonlinlstsq} is not needed here: Lemma~\ref{lem:MB} and Proposition~\ref{prop:picontinuous} suffice, as detailed early in the proof of Proposition~\ref{prop:Ssmoothpisubmersion}.}
In particular, $S$ cannot be diffeomorphic to a closed ball, sphere, or cylinder.
Moreover, since the only compact contractible smooth manifolds are singletons---this is a well-known fact, see Appendix~\ref{app:compactimpliespoint}---we obtain the following (also independently shown by~\citet{bennejma2025PLcompact}).
\begin{corollary}[compact $\implies$ point] \label{cor:compactSsingleton}
  Assume $\calM$ is contractible.
  Let $f \colon \calM \to \reals$ be smooth and globally \pl{}.
  If its set of minimizers $S$ is compact, then $S$ is a singleton.
  In particular, the conclusions of Theorem~\ref{thm:globalsquaredistsinglecppl} apply to $f$.
\end{corollary}

Additional considerations lead us to a complete characterization of the possible sets $S$ that can arise.
Theorem~\ref{thm:plnonlinlstsq} shows that if $\calM$ is contractible and admits a smooth, globally \pl{} function with minimizing set $S$, then there exists a diffeomorphism $\psi \colon \calM \to S \times \Rk$ satisfying $\psi(S) = S \times \{0\}$.
The next theorem provides a converse, and in fact does not require $\calM$ to be contractible. 
See Section~\ref{sec:corollaries} for a proof.

\begin{theorem}[Constructing globally \pl{} functions]\label{thm:constructingglobalPLcor}
  Let $S$ be a smooth embedded submanifold of a smooth manifold $\calM$.
  Suppose there exists a diffeomorphism
  \begin{align*}
    \psi \colon \calM \to S \times \reals^k && \text{satisfying} && \psi(S) = S \times \{0\}. 
  \end{align*}
  Then, for every Riemannian metric on $\calM$, there exists a smooth, globally \pl{} function $f \colon \calM \to \reals$ whose set of minimizers is $S$.
\end{theorem}

Together, Theorems~\ref{thm:plnonlinlstsq} and~\ref{thm:constructingglobalPLcor} (with the help of the classical Theorem~\ref{thm:deepthm2}) provide the sought characterization of $S$.

\begin{corollary}[Characterization of $S$, up to diffeomorphism]\label{cor:equivonS}
Let $\tilde{S}$ be a smooth manifold.
Fix $n > \dim(\tilde{S})$, and endow $\calM = \reals^n$ with a complete Riemannian metric $\inner{\cdot}{\cdot}$ (not necessarily the Euclidean one).
The following are equivalent:
\begin{itemize}
\item[(a)] $\tilde{S}$ is diffeomorphic to the minimizer set $S$ of a smooth function $f \colon \calM \to \reals$ which is globally \pl{} with respect to the given metric $\inner{\cdot}{\cdot}$.
\item[(b)] $\tilde{S}$ is contractible.
\end{itemize}
\end{corollary}
\begin{proof}
  To show (a) implies (b), assume there exists a smooth globally \pl{} function $f \colon \calM \to \reals$ whose minimizer set $S$ is diffeomorphic to $\tilde{S}$.
  Since $\calM$ is contractible, so is $S$ (appeal to Theorem~\ref{thm:plnonlinlstsq} or Proposition~\ref{prop:picontinuous}), and therefore so is $\tilde{S}$.

  To show (b) implies (a), assume $\tilde{S}$ is contractible, and let $k = n - \dim(\tilde{S}) \geq 1$.
  By Theorem~\ref{thm:deepthm2}, there is a diffeomorphism $\tilde{\psi} \colon \reals^n \to \tilde{S} \times \reals^k$.
  Let $S = (\tilde{\psi})^{-1}(\tilde{S} \times \{0\})$.
  Since the restriction of a diffeomorphism remains a diffeomorphism, $S$ is a submanifold of $\reals^n$ and it is diffeomorphic to $\tilde{S}$.
  In particular, composing diffeomorphisms, we find there is a diffeomorphism $\psi \colon \reals^n \to S \times \reals^k$ such that $\psi(S) = S \times \{0\}$.
  Theorem~\ref{thm:constructingglobalPLcor} therefore implies the existence of a smooth globally \pl{} function $f \colon \calM \to \reals$ with minimizer set $S$.
\end{proof}

Beyond characterizing $S$ up to diffeomorphism, we may also study its possible \emph{embeddings} in $\calM$.
Theorem~\ref{thm:plnonlinlstsq} provides even more information in that regard.
For example, it rules out $S$ being a knotted line in $\reals^3$ (also known as a \emph{long knot}).
See Appendix~\ref{app:knottheory} for details.

\subsubsection{When is $f$ a pure quadratic, up to a change of variable?}\label{sec:SeqRk}

According to Theorem~\ref{thm:globalsquaredistsinglecppl}, if $S$ is a singleton, then $f$ can be deformed into a pure quadratic:
\begin{equation*}
  f(\varphi^{-1}(y)) \;=\; f^* + \|y\|^2, \quad \quad \forall y \in \reals^n.
\end{equation*}
A natural question is whether the same holds when $S$ is not a singleton.  
Specifically, given a globally \pl{} function $f \colon \calM \to \reals$, does there exist a diffeomorphism $\varphi \colon \calM \to \reals^n$ such that
\begin{equation} \label{eq:purequadraticQ}
  f(\varphi^{-1}(y)) \;=\; f^* + y_{m+1}^2 + \cdots + y_n^2, \qquad \forall y \in \reals^n,
\end{equation}
where $m = \dim S$?
If so, then
$
  S = \varphi^{-1}(\{y \in \reals^n : y_{m+1} = \cdots = y_n = 0\})
$
is diffeomorphic to $\reals^m$.
Is that always the case?

The answer is negative \emph{even if} $\calM = \Rn$.
In fact, there exist globally \pl{} functions whose sets of minimizers are \emph{not even homeomorphic} to any linear space, as we now explain.

In light of Corollary~\ref{cor:equivonS}, the question reduces to the following:
do there exist contractible smooth manifolds $S$ that are not homeomorphic to a Euclidean space?  
If $\dim S \leq 2$, the answer is no (Theorem~\ref{thm:deepthm1}).  
Beginning in dimension three, however, such manifolds exist.  
The first example was discovered by~\citet{Whiteheadmanifold1935}.
He had previously claimed that no such example could exist, but in the course of correcting this error he constructed the counterexample, now known as the \emph{Whitehead manifold}---see~\citep{wildwildwhitehead2019} for an exposition.
Subsequently, Mazur and others produced further examples~\citep{Mazurmanifolds1961}.  
The essential obstruction is that these manifolds are not \emph{simply connected at infinity} (see Remark~\ref{rem:simplyconnectedatinfinity}).

As a result, we have the following consequence of Corollary~\ref{cor:equivonS}.
\begin{corollary}[\pl{} functions with $S \not\cong \reals^m$] \label{cor:whitehead}
For every $m \geq 3$ and $n \geq m+1$, 
there exists a smooth globally \pl{} function on $\calM = \reals^n$ (with the Euclidean metric)
whose set of minimizers $S$ is an $m$-dimensional submanifold that is not homeomorphic to $\reals^m$.
\end{corollary}
\begin{proof}
Choose a contractible smooth $m$-dimensional manifold $\tilde{S}$ not homeomorphic to a linear space~\citep{Whiteheadmanifold1935,Mazurmanifolds1961}.
Invoke Corollary~\ref{cor:equivonS}.
\end{proof}

On the other hand, if we assume $S$ is diffeomorphic to a Euclidean space, then so is $\calM$, and $f$ can indeed be deformed into a pure quadratic.
\begin{corollary}[$f$ deforms to a quadratic iff $S \cong \reals^m$] \label{cor:deformintoquadform}
  Let $f \colon \calM \to \reals$ be smooth and globally \pl{}.
  If (and only if) its set of minimizers $S$ is diffeomorphic to $\reals^m$, there exists a diffeomorphism $\xi \colon \calM \to \reals^n$ such that
  \begin{align*}
    f(\xi^{-1}(y)) \;=\; f^* + y_{m+1}^2 + \cdots + y_n^2, \quad \quad \forall y \in \reals^n.
  \end{align*}
\end{corollary}
\begin{proof}
  The ``only if'' part is clear: see the comment after~\eqref{eq:purequadraticQ}.
  Now for the other direction:
  since $S$ is contractible, so is $\calM$ (Proposition~\ref{prop:picontinuous}).
  Thus, Theorem~\ref{thm:plnonlinlstsq} provides a diffeomorphism $\psi \colon \calM \to S \times \reals^k$ of the form $\psi = (\pi, \varphi)$ such that $f(x) = f^* + \|\varphi(x)\|^2$ for all $x \in \calM$.
  Choose a diffeomorphism $\sigma \colon S \to \reals^m$ and let $\xi(x) := (\sigma(\pi(x)), \varphi(x))$.
  This is a diffeomorphism from $\calM$ to $\reals^m \times \reals^k = \reals^n$ by composition, and if $y = \xi(x)$, then $f(\xi^{-1}(y)) = f^* + \|\varphi(x)\|^2$ and $\varphi(x) = (y_{m+1}, \ldots, y_n)$.
\end{proof}

In particular, if $\calM$ is contractible, then $S$ is diffeomorphic to $\reals^m$ (and hence Corollary~\ref{cor:deformintoquadform} applies) whenever $\dim S \leq 2$, and also when $\dim S = 3$ or $\dim S \geq 5$ provided $S$ is simply connected at infinity: this follows from the (classical) Theorem~\ref{thm:deepthm1} (in appendix), which is a consequence of work by \citet{stallings1962piecewise}, \citet{HuschPrice1970} and \citet{Perelman20022003}.

Even if $f$ itself cannot be deformed to a pure quadratic, the ``lifted'' function
\begin{align}\label{eq:eqnforgextended}
g \colon \calM \times \reals \to \reals, \quad \quad g(x, t) = f(x),
\end{align}
which is also \pl{} (in the product metric) and has minimizer set $S \times \reals$, can always be deformed to a pure quadratic, provided $\calM$ is contractible.
\begin{corollary}
  Let $f \colon \calM \to \reals$ be smooth and globally \pl{} with minimizer set $S$.
  Define $g$ as in~\eqref{eq:eqnforgextended} and let $m = \dim S$.
  If (and only if) $\calM$ is contractible, there exists a diffeomorphism $\xi \colon \calM \times \reals \to \reals^{n+1}$ satisfying
  \begin{align*}
    g(\xi^{-1}(y)) = f^* + y_{m+2}^2 + \cdots + y_{n+1}^2, \quad \quad \forall y \in \reals^{n+1}.
  \end{align*}
\end{corollary}
\begin{proof}
  The ``only if'' part is trivial: if $\calM \times \reals$ is diffeomorphic to a linear space, then it is contractible; and it is homotopy equivalent to $\calM$ hence $\calM$ is contractible.
  The ``if'' part is a consequence of Corollary~\ref{cor:deformintoquadform} (applied to $g$) and of the classical Theorem~\ref{thm:deepthm2} applied to $S \times \reals$ (the set of minimizers of $g$).
\end{proof}

\subsubsection{Globally \pl{} functions are geodesically convex in some metric}\label{sec:gconvexintro}

A function $f \colon \calM \to \reals$ is said to be \emph{geodesically convex} (or \emph{g-convex}) if, along every geodesic segment $\gamma \colon [0,1] \to \calM$, the composition $f \circ \gamma$ is convex in the usual one-dimensional sense, that is,
\[
  f(\gamma(t))
  \;\leq\;
  (1-t) f(\gamma(0)) + t f(\gamma(1))
  \qquad\forall\, t \in [0,1].
\]
For an overview in the context of optimization on manifolds, see \citet{udriste1994convex} or~\citep[Ch.~11]{boumal2020intromanifolds}.

It has been asked\footnote{A similar question was studied by~\citet[Thm.~5.1]{rapcsak1993coordinatetransformations} for the case where $S$ is a singleton (though it was not clear to us how to interpret their proof) and by~\citet[p.~295]{udriste1994convex} for the case $\calM = \reals$ (dimension one).} whether globally \pl{} functions on $\Rn$ are ``convex in disguise'', in the sense that they are g-convex with respect to some Riemannian metric.
We show that this is indeed the case, as a consequence of Theorem~\ref{thm:plnonlinlstsq}.


\begin{theorem}[\pl{} $\implies$ g-convex in some metric] \label{thm:pl-gconvex}
  Let $f \colon \calM \to \reals$ be smooth and globally \pl{} with respect to some complete Riemannian metric $\inner{\cdot}{\cdot}_1$.
  The set $S$ of minimizers of $f$ is a smooth embedded submanifold of $\calM$ with $\dim S = m$.
  \begin{enumerate}[label=(\alph*)]
  \item If $\calM$ is contractible, then it admits a complete Riemannian metric $\inner{\cdot}{\cdot}_2$ such that $f$ is geodesically convex and globally 1-\pl{} with respect to $\inner{\cdot}{\cdot}_2$. \label{item:pl-gconvex-basic}
  \item If (and only if) $S$ is diffeomorphic to $\reals^m$, the metric $\inner{\cdot}{\cdot}_2$ in part~\ref{item:pl-gconvex-basic} can be chosen such that $\calM$ is isometric to Euclidean space. \label{item:pl-gconvex-S-Rm}
  \end{enumerate}
\end{theorem}
\begin{proof}[Proof sketch]
  The essence of the argument is to invoke Theorem~\ref{thm:plnonlinlstsq} to obtain a diffeomorphism $\psi \colon \calM \to S \times \Rk$ and to give $S \times \Rk$ a complete metric such that $f \circ \psi^{-1}$ is globally \pl{} and geodesically convex in that metric.
  The function $f$ then inherits those qualities by pulling back the metric to $\calM$ through $\psi$.
  See Section~\ref{sec:gconvexproof} for details.
\end{proof}

\noindent
Note that part~\ref{item:pl-gconvex-basic} is not ``if and only if'': consider Example~\ref{ex:cylinder} which exhibits a g-convex and globally \pl{} function on a cylinder.
Also, part~\ref{item:pl-gconvex-S-Rm} is a variation of Corollary~\ref{cor:deformintoquadform}.

\subsection{Related work} \label{sec:relatedwork}

The literature related to the \pl{} condition and to the type of results we obtain is vast and has deep roots.
We organize some of it in various categories below, without repeating the pointers given above, or anticipating in-context references given throughout the paper.

\paragraph{Optimization algorithms.}

The original appeal of globally \pl{} functions in optimization is that they allow for linear convergence of gradient descent to the set of minimizers.
This was first shown by \citet{polyak1963gradientmethods}.
Moreover, it appears that the \pl{} condition is essentially necessary to guarantee such rates of convergence with constant-step-size gradient descent in the nonconvex case~\citep{abbaszadehpeivasti2022conditions}, although adaptive step sizes can help~\citep{davis2024gradientdescentadaptivestepsize}. 
Beyond gradient descent, many algorithms enjoy rapid convergence rates under \pl{}, including cubic regularization~\citep{nesterov2006cubic}, coordinate descent and stochastic gradient methods~\citep{karimi2016linearPL}, and trust-region methods with truncated conjugate gradients~\citep{rebjock2023nonisolatedTRtCG}.
The \pl{} assumption is strictly weaker than strong convexity, and it was shown that there exists a complexity separation between those two classes~\citep{yue23lowerboundPL}.

\paragraph{\pl{} in applications.}

Optimization problems whose cost functions satisfy the \pl{} condition globally (or in large regions) have been observed in various applications.
In a retrospective article about some of Polyak's work, \citet[\S3]{ablaev2024polyak} list three sorts of examples:
\begin{itemize}
  \item The usual nonlinear least squares, as in Section~\ref{sec:particularcaselstsq}.
  \item Cases where $f(x) = g(Ax + b)$ with a strongly convex function $g$~\citep{karimi2016linearPL}, which they extend to strictly convex $g$ to encompass logistic regression (at the cost of having \pl{} in a wide region rather than being truly global).
  \item Optimal control problems~\citep{fatkhullin2021linearfeedback}.
\end{itemize}
Also in control, \citet[\S3]{fazelLQR2018} show that the optimization problem underlying the model-free linear quadratic regulator (LQR) satisfies a global \pl{} condition wherever the objective is finite, which in turn yields efficient sample and computational guarantees for learning the regulator.
See also remarks by \citet{deoliveira2025remarksPL} about the possibility of having \pl{} in continuous-time versus discrete-time LQR.

Another example is the computation of Bures--Wasserstein barycenters: although the objective is not geodesically convex, \citet{chewi20a} show that it satisfies a global \pl{} condition, which they exploit to secure fast optimization.

Many more examples can be found where the cost function $f$ satisfies the \pl{} condition \emph{locally} around $S$.
In machine learning, this appears (with various tweaks) in works about the loss landscapes of overparameterized neural networks~\citep{belkinPLstar}, how data is processed by deep transformers~\citep{chen2025quantitativeclusteringmeanfieldtransformer}, the analysis of gradient descent for deep networks~\citep{chatterjee2022convergencegradientdescentdeep}, and more~\citep{aimingBelkin,lucchiPL2024}. 
Beyond neural networks, local \pl{} arises in problems that are reparameterizations of simpler ones, such as the low-rank desingularization~\citep{rebjock2024desingularization}.
Variations of local \pl{} have also found applications in queueing theory~\citep{chen2025hiddencvxqueue} (with box constraints), as well as sampling, due to its connection with the log-Sobolev inequality~\citep{chewi2024ballisticlimitlogsobolevconstant} and the Poincar\'e inequality~\citep{gong2025poincareinequalitylocallogpolyakl}.

\paragraph{Similar structural results on $f$.}

Classically, the Morse lemma shows that a smooth function is \emph{locally} equivalent (up to a change of variable) to a quadratic form near nondegenerate critical points.
The Morse--Bott lemma extends this to critical manifolds, similarly to the Morse lemma with parameters which yields smoothly varying quadratic forms.

Theorem~\ref{thm:plnonlinlstsq} provides a diffeomorphism $\psi$ such that $(f \circ \psi^{-1})(x, v) = f^* + \|v\|^2$ is quadratic.
This is akin to a \emph{global} Morse--Bott lemma, afforded by our global assumptions.

Theorem~\ref{thm:globalsquaredistsinglecppl} follows from Theorem~\ref{thm:globalsquaredistsinglecp}, which requires the Hessian at $x^*$ to be positive definite.
This can be relaxed: see for example a result by \citet[Prop.~1]{grune1999stable}, updated by \citet[Prop.~2]{kvalheim2025stablewohyperbolic} to reflect the resolution of the Poincar\'e conjecture.
In contrast to Theorem~\ref{thm:globalsquaredistsinglecp}, these (more general) results only provide a $C^1$ homeomorphism that restricts to a diffeomorphism upon removing $x^*$.
Their proofs rely on advanced topological results that limit applicability in dimension 5.
These differences allow us to emphasize the importance of the initial step in our proof of Theorem~\ref{thm:globalsquaredistsinglecp}, where we first locally straighten the landscape via the Morse Lemma, as depicted in Figure~\ref{fig:flowlines}.

In a similar spirit but for non-isolated critical points, \citet[Thm.~11, Cor.~20, Rem.~7]{kvalheim2025difftopology} proves a related result for a more general setup in dynamical systems.
One could try to apply it to the end-point map $\pi$ once it is known to be a trivial smooth fiber bundle (as we show by the end of Section~\ref{sec:endpointmap}).
As stated, Kvalheim's result assumes $S$ is compact, which in our case forces $S$ to be a singleton (Corollary~\ref{cor:compactSsingleton}).
However, that assumption could conceivably be lifted with localized changes.
If so, Kvalheim's general techniques would provide a map $\psi$ akin to the one we build in Theorem~\ref{thm:plnonlinlstsq}, with one important caveat: $\psi$ would be a diffeomorphism away from $S$, but only a homeomorphism when including $S$.
We explored many proof techniques for Theorem~\ref{thm:plnonlinlstsq}, and the one we present in this paper is the only one we could find that yields a truly global diffeomorphism.
Overall, our proof techniques are rather different, relying on a transparent construction of the trivial fiber bundle structure in Section~\ref{sec:fiberbundle}.

For the construction of the fiber bundle structure of $\pi$ itself, we were also hoping to rely more heavily on existing literature.
However, we could not find a result that applies to our setting.
For example, the work of \citet{Eldering2018NAIM} comes close, but the main results in that paper (and many others) imposes compactness assumptions, with the associated issues as outlined in the previous paragraph.

Also recently, \citet[Thm.~3.9]{marteau2024second} study smooth nonnegative functions (not necessarily \pl{}) whose sets of minimizers are smooth manifolds.
Under a type of Morse--Bott condition (akin to local \pl{}), they prove that such functions admit a global decomposition as a sum of squares of \emph{countably many} smooth functions.
(They also provide additional context and motivation for decomposing functions in sums of nonlinear squares.)
In contrast, Theorem~\ref{thm:plnonlinlstsq} shows that if $f$ is globally \pl{} (a stronger assumption), it can be decomposed globally as a sum at most $n$ squares, with additional structure that enables many of the corollaries in Section~\ref{sec:implications}.

To go beyond quadratic growth (see Lemma~\ref{lemma:bound-trajectories-QG}), \citet{davis2024gradientdescentadaptivestepsize} show that any smooth function satisfying a fourth-order growth condition around its minimizers admits a \emph{local} ``ravine'' decomposition: the function splits into tangent and normal components, with quadratic growth in normal directions and slower growth along the tangent directions. 
Their proof relies on the Morse lemma with parameters.

\citet{garrigos2023squaredistancefunctionspolyaklojasiewicz} shows that squared distance functions to arbitrary closed sets are globally \pl{}, and conversely that any \pl{} function admits a \emph{lower bound} by such a squared distance to its minimizer set.
Such functions are not necessarily smooth (for example, the squared distance to the interval $[-1, 1]$ on the real line is $f(x) = \max(0, |x|-1)^2$, which is \pl{} and $C^1$ but not $C^2$).
Accordingly, the paper caters to a nonsmooth variant of the \pl{} condition, replacing the gradient with a limiting subdifferential.

Let us also note that, for $C^2$ functions, the \emph{local} \pl{} property is equivalent to other \emph{local} properties such as quadratic growth, Morse--Bott, error bound, and the restricted secant inequality~\citep{rebjock2023nonisolated}.

\paragraph{Similar structural results on $\calM$.}

As shown in Theorem~\ref{thm:globalsquaredistsinglecppl} and Corollary~\ref{cor:deformintoquadform}, the mere existence of an appropriate globally \pl{} function on $\calM$ can be used to infer that $\calM$ is diffeomorphic to $\Rn$.
We already noted after Theorem~\ref{thm:globalsquaredistsinglecppl} how this relates to a result of \citet{brown1961monotoneunion}.

Similarly, \citet[Prop.~5.10]{sakai1996riemannian} shows (crediting \citet{greene1976}) that the existence of a smooth function $f \colon \calM \to \reals$ that is strongly g-convex (and in particular, coercive) implies that $\calM$ is diffeomorphic to $\Rn$.
If the Hessian of $f$ is (almost) identity, then $\calM$ is (almost) isometric to $\Rn$ (see \citep{kasue1981strictconvex} and the 1979 PhD thesis of H.W.\ Wissner).
If $f$ is merely g-convex and its set of minimizers $S$ has no boundary, then $\calM$ is diffeomorphic to the total space of the normal bundle of $S$, akin to our conclusion in Theorem~\ref{thm:plnonlinlstsq}---see \citep[p.~438]{shiohama1984topocompletemanifold}.

On a different note, Theorem~\ref{thm:globalsquaredistsinglecp} implies that if $f \colon \calM \to \reals$ is a coercive Morse function with a \emph{single} critical point then $\calM$ is diffeomorphic to $\Rn$.
This is similar in spirit to (albeit simpler than) Reeb's sphere theorem which states that if $f \colon \calM \to \reals$ is a Morse function with exactly \emph{two} critical points and $\calM$ is compact then $\calM$ is homeomorphic to a sphere~\citep[Thm.~4.1]{milnor1963morse}.

\paragraph{Similar structural results on $S$.}

Gradient inequalities have long been used to infer topological properties of level sets. 
In particular, \citet{lojasiewicz1963propriete} introduced his inequality to study zero sets 
of real-analytic functions, showing that gradient flow induces deformation retractions onto these sets; see also \citep[Prop.~3]{kurdyka1998gradientsomin}. 

In a related direction, \citet{cibotaru2025mappingcylinder} investigate the topological structure of the zero set of a function satisfying a Kurdyka--\L{}ojasiewicz inequality.
Working under weaker regularity assumptions than those adopted in the present paper, they show that the zero set admits a \emph{regular mapping-cylinder neighborhood} that is invariant under negative gradient flow. 
This strengthens earlier results of~\citet{kurdyka1998gradientsomin}, who established that the zero set is a strong deformation retract of a suitable neighborhood. 
As an application, they derive restrictions on the types of embedded subsets that can arise as zero sets of K\L{} functions, ruling out certain wild embeddings.

\paragraph{Function classes related to \pl{}.}

A function $f$ is \emph{invex} if its stationary points are its global minimizers.
Convex functions and globally \pl{} functions are invex.
Many more subclasses of invex functions are studied and compared to each other by \citet{goujaud2022conditionnumbers} and \citet{sidford2018quasar}.

The \pl{} condition is a particular case of the more general \emph{{\L}ojasiewicz inequality}.
\citet{lojasiewicz1963propriete,lojasiewicz1965ensembles,lojasiewicz1982trajectoires} proved that every real-analytic function satisfies that property locally.
This was later generalized to the \emph{Kurdyka--{\L}ojasiewicz (K{\L}) property}, involving a desingularizing function $\sigma$, and leading to inequalities of the form $\sigma\big(f(x) - f^*\big) \;\leq\; c \|\nabla f(x)\|.$
This framework, introduced by \citet{kurdyka1998gradientsomin} and further developed by \citet{attouch2010proximal}, \citet{bolte2010characterizations}, \citet{lewis2024identifiability} and \citet{li2018calculusKL} among others, allows to go well beyond smooth $f$.

Another structural assumption that has recently received attention is
\emph{hidden convexity}, whereby a nonconvex objective becomes convex
after a nonlinear \emph{invertible} change of variables.
This setting has been explored in stochastic and constrained optimization by~\citet{fatkhullin2024hiddenconvexstoch,fatkhullin2025hiddenconvexconstrained},
who show that such structure enables convex-like global convergence guarantees
for first-order methods even when the convex reformulation is not explicitly
available.  
More generally, nonlinear reparameterizations---possibly noninvertible---that transform nonconvex problems into convex ones are studied by \citet{levin2022lifts} and the references therein.

\section{Basic facts about P{\L} functions and topological notions} \label{sec:basicfacts}

Let us open these preliminaries with the classical connection between \pl{} and \emph{quadratic growth}.

\begin{lemma}[Bounded trajectories and quadratic growth] \label{lemma:bound-trajectories-QG}
  Let $f \colon \calM \to \reals$ be continuously differentiable and globally $\mu$-\pl{}.
  Let $x(t)$ be the negative gradient flow trajectory of $f$ (that is, $x'(t) = -\grad f(x(t))$) initialized at $x(0) = x_0 \in \calM$.
  Then, $x(t)$ is well defined for all $t \geq 0$ and the trajectory has bounded length for $t \in [0, \infty]$.\footnote{Trajectories may not be defined for all $t < 0$, as shown by $f(x) = \log\!\big(e^{x^2} + e^{x^4}\big)$, which is \pl{}.}
  In particular, the trajectory converges to some $x_\infty := \lim_{t\to\infty} x(t) \in \calM$, and
  \begin{align*}
    \dist\!\big( x_0, x_\infty \big) \leq \sqrt{\frac{2}{\mu}\Big(f(x_0) - f^*\Big)},
  \end{align*}
  where $f^* = \inf f$ and $\dist$ is the Riemannian distance.
  The limit point $x_\infty$ is a critical point of $f$.
  Therefore, the set of critical points $S$~\eqref{eq:S} is closed and non-empty, and
  \begin{align}
      f(x) - f^* \geq \frac{\mu}{2} \dist(x, S)^2
      \tag{QG}
      \label{eq:QG}
  \end{align}
  for all $x \in \calM$, where $\dist(x, S) := \inf_{y \in S} \dist(x, y)$.
\end{lemma}
\begin{proof}
  See the classical argument in~\citep[Prop.~1']{otto2000generalization}, and broader historical notes in~\citep[Lem.~A.1]{rebjock2023nonisolated}.
  The proof that negative gradient flow trajectories on $f$ have bounded length parallels the argument used earlier by \citet[Thm.~1]{lojasiewicz1982trajectoires} for analytic functions.
  The set $S = f^{-1}(f^*)$ is closed, and it is non-empty because it contains the limit points of all trajectories.
\end{proof}

The next lemma underlines the relation between \pl{} and the \emph{Morse--Bott property}.
The most important aspect of it for our purposes is that $f$ being \pl{} and smooth implies that $S$ is (locally) smooth.
This was known for analytic functions ($C^\omega$) and recently extended to $C^2$ functions (it does not hold if $f$ is merely $C^1$).

\begin{lemma}[Morse--Bott property]\label{lem:MB}
    Let $f \colon \calM \to \reals$ be globally $\mu$-\pl{} and let $S$ be its set of critical points.
    If $f$ is \{$C^{p+1}$ with $p \geq 1$, $C^\infty$ or $C^\omega$\}, then
    \begin{enumerate}
        \item Each connected component of $S$ is a \{$C^{p}$, $C^\infty$ or $C^\omega$\} embedded submanifold of $\calM$.
        \item For each $x$ in $S$, let $\T_x S$ and $\N_x S$ denote the tangent and normal spaces at $x$ to the corresponding connected component of $S$.
        The Hessian of $f$ at $x$ satisfies:
        \begin{align}
            \ker \Hess f(x) = \T_x S && \textrm{ and } && \Hess f(x)|_{\N_x S} \succeq \mu \Id,
            \tag{MB}
            \label{eq:MB}
        \end{align}
        where $\Id$ denotes the identity operator (here on the normal space $\N_x S$).
    \end{enumerate}
\end{lemma}
\begin{proof}
    See~\citep[Thm.~2.16, Cor.~2.17]{rebjock2023nonisolated} for regularity $C^p$ and $C^\infty$, and~\citep{feehan2020morse} for the analytic case.
    A local version of \pl{} is sufficient.
\end{proof}

As stated, Lemma~\ref{lem:MB} does not imply that $S$ itself is a manifold in the usual sense since, in principle, it might have several connected components of different dimensions.
We will rule this out shortly, by showing in Proposition~\ref{prop:picontinuous} that $S$ is connected because $\calM$ is so.

The latter proposition shows something more, namely, that $\calM$ strongly deformation retracts to $S$.
This is why $S$ and $\calM$ share many topological properties.
In particular, Theorem~\ref{thm:plnonlinlstsq} assumes $\calM$ is contractible, and we shall see that this is the case if and only if $S$ is contractible.
Let us recall the definitions~\citep[pp.~200--202]{lee2011topological}. 

\begin{definition} \label{def:deformationretract}
    Let $X$ be a topological space.
    A continuous map $H \colon X \times [0, 1] \to X$ is a \emph{deformation retraction} of $X$ to a topological subspace $A \subseteq X$ if, for all $x \in X$ and $a \in A$,
    \begin{align*}
      H(x, 0) = x, && H(x, 1) \in A, && \textrm{ and } && H(a, 1) = a.
    \end{align*}
    We then say \emph{$X$ deformation retracts to $A$}.
    If also $H(a, t) = a$ for all $a \in A$ and $t \in [0, 1]$, then $H$ is a \emph{strong deformation retraction}.
\end{definition}

\begin{definition} \label{def:contractible}
    A topological space $X$ is \emph{contractible} if it deformation retracts onto a point.
\end{definition}

Parts of our conclusions are that the end-point map of negative gradient flow $\pi \colon \calM \to S$ (see~\eqref{eq:pi} below) is a smooth fiber bundle---a trivial one if $\calM$ is contractible.
The definition follows~\citep[p.~268]{lee2012smoothmanifolds}.

\begin{definition} \label{def:smoothfiberbundle}
    Let $E, B$ and $F$ be smooth manifolds, and let $\pi \colon E \to B$ be surjective and smooth. 
    Then $\pi$ is a \emph{smooth fiber bundle over the base $B$ with model fiber $F$} if, for all $b \in B$, there exist a neighborhood $U$ of $b$ in $B$ and a diffeomorphism $h \colon \pi^{-1}(U) \to U \times F$ such that, for all $x$, the first component of $h(x)$ is $\pi(x)$.

    Such a map $h$ is called a \emph{local trivialization}.
    If it can be made global, that is, if there exists a diffeomorphism $h \colon E \to B \times F$ such that for all $x$ the first component of $h(x)$ is $\pi(x)$, then the fiber bundle is said to be (globally) \emph{trivial}.
\end{definition}

Note that the definition implies $\pi$ is a submersion, that is, $\D\pi(x)$ is surjective for all $x$. 

\section{The special case of a single minimizer} \label{sec:singleminimizer}

\begin{figure}[t]
    \centering
    \includegraphics[width=1.0\textwidth]{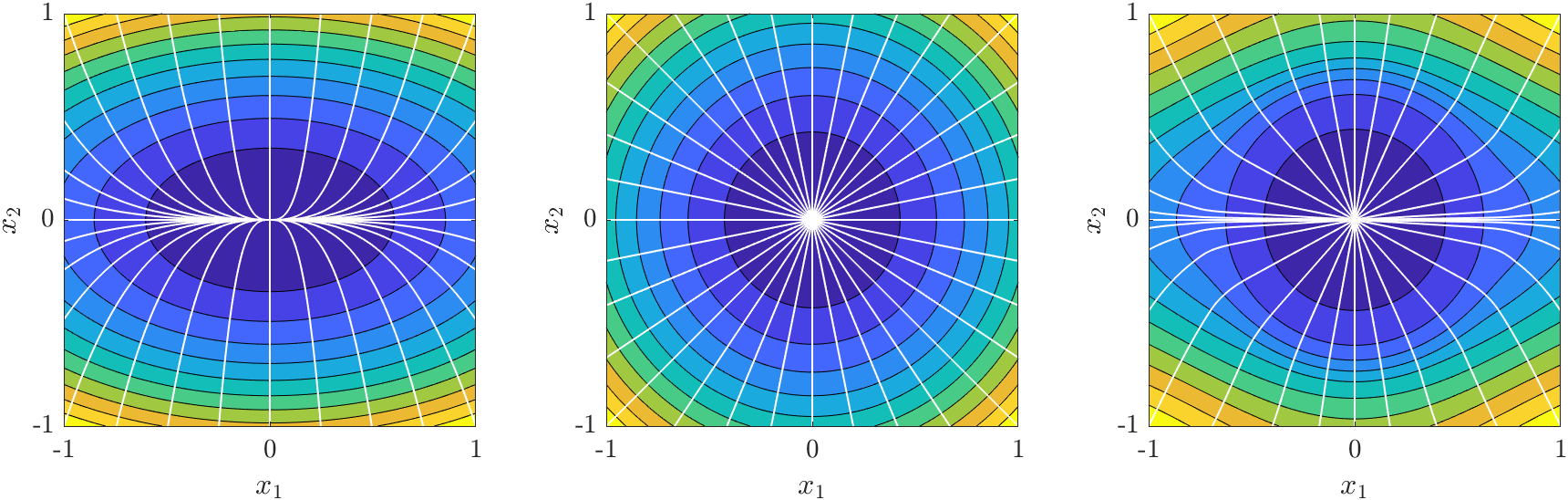}
    \caption{
        Background colors indicate level sets of some function $f \colon \reals^2 \to \reals$ with minimizer at the origin $x^* = 0$.
        White curves are gradient flow trajectories.
        \emph{Left:} $f$ is quadratic; the eigenvalues $\lambda_1, \lambda_2$ of the Hessian of $f$ at $x^*$ are positive but distinct. We see that most trajectories approach $x^*$ from the same direction, asymptotically.
        \emph{Middle:} if instead the two eigenvalues are equal, then the trajectories approach $x^*$ from all directions.
        \emph{Right:} to prove Theorem~\ref{thm:globalsquaredistsinglecp}, we first deform space so that the eigenvalues of the Hessian of $f$ (not necessarily quadratic) at $x^*$ become equal. This helps match trajectories to directions in $\Rn$ in a smooth way.
    }
    \label{fig:flowlines}
\end{figure}

This section holds a proof of Theorem~\ref{thm:globalsquaredistsinglecppl}, that is, the case where $f$ has a single minimizer.
In fact, we shall prove a somewhat more general result as stated in Theorem~\ref{thm:globalsquaredistsinglecp} below.
It relaxes the global \pl{} assumption to a local version of it together with \emph{coercivity}.

The proof relies on two lemmas established later in the section.
Recall the goal is to build a diffeomorphism $\varphi \colon \calM \to \Rn$ such that $f(x) = f(x^*) + \|\varphi(x)\|^2$ for all $x \in \calM$.
\begin{enumerate}
    \item Lemma~\ref{lem:localsquaredistsinglecp} is a globalized Morse lemma.
    We use it to transform $\calM$ globally (via a diffeomorphism) in such a way that, locally around the critical point $x^*$, the function becomes equal to the squared distance to $x^*$.
    
    This sets the stage for the next step, as it ensures that negative gradient flow trajectories can converge to $x^*$ arriving from all directions (asymptotically), whereas before the transformation the trajectories might collapse according to the extreme eigendirections of the Hessian at $x^*$.
    See Figure~\ref{fig:flowlines}.

    \item Lemma~\ref{lem:globalquadraticsinglecp} uses gradient flow on $f$ to map each point of $\calM$ to a point of $\Rn$, diffeomorphically.
    To do so, we look at the direction of arrival of the gradient flow trajectory as it converges to $x^*$.
    This provides a direction in $\Rn$, which is then scaled to ``rectify'' the function value into a pure quadratic.
    The proof relies on a helper Lemma~\ref{lem:GFrescaledoncoerciveinvexM} about a normalized gradient flow that maps level sets to level sets.
\end{enumerate}

\begin{definition} \label{def:coercive}
    A function $f \colon \calM \to \reals$
    is \emph{coercive}\footnote{Coercive functions are also called \emph{exhaustion} functions~\citep[p.~46]{lee2012smoothmanifolds}.
A coercive function is \emph{proper} (that is, pre-images of compact sets are compact).
The other way around, a proper function that is also lower-bounded is coercive. 
} if its sublevel sets are compact, that is, if for all $c \in \reals$ the set $\{ x \in \calM : f(x) \leq c \}$ is compact.
\end{definition}

\begin{lemma} \label{lem:coerciveuniquecpglobalmin}
    If $f \colon \calM \to \reals$ is smooth and coercive and it has a unique critical point $x^*$, then $x^*$ is the unique global minimizer of $f$.
\end{lemma}
\begin{proof}
    The sublevel set $\{ x \in \calM : f(x) \leq f(x^*) \}$ is compact, hence $f$ has a global minimizer.
    Moreover, any global minimizer of $f$ must be a critical point.
\end{proof}

\begin{theorem} \label{thm:globalsquaredistsinglecp}
    Let $f \colon \calM \to \reals$ be smooth and coercive.
    Assume $f$ has a unique critical point $x^*$ and that the Hessian of $f$ at $x^*$ is positive definite.
    Then, there exists a diffeomorphism $\varphi \colon \calM \to \Rn$ such that $f(x) = f(x^*) + \|\varphi(x)\|^2$ for all $x \in \calM$.
\end{theorem}
\begin{proof}
    From Lemma~\ref{lem:coerciveuniquecpglobalmin}, we know that $x^*$ is the unique minimizer of $f$.
    Apply Lemma~\ref{lem:localsquaredistsinglecp} to $f$ to obtain a diffeomorphism $\psi \colon \calM \to \calM$ with the stated properties.
    Then apply Lemma~\ref{lem:globalquadraticsinglecp} to $f \circ \psi$, which yields a diffeomorphism $\tilde \varphi \colon \calM \to \Rn$ such that $(f \circ \psi)(x) = f(x^*) + \|\tilde \varphi(x)\|^2$ for all $x \in \calM$.
    The composition $\varphi = \tilde \varphi \circ \psi^{-1}$ is the desired diffeomorphism because $f(x) = (f\circ\psi)(\psi^{-1}(x)) = f(x^*) + \|\tilde\varphi(\psi^{-1}(x))\|^2$.
\end{proof}

From here, Theorem~\ref{thm:globalsquaredistsinglecppl} as stated in the introduction is a corollary.

\begin{proof}[Proof of Theorem~\ref{thm:globalsquaredistsinglecppl}]
    Recall $f \colon \calM \to \reals$ is smooth and globally~\eqref{eq:PL} with a unique critical point $x^*$.
    The sublevel sets of $f$ are compact (and hence $f$ is coercive) owing to quadratic growth away from $x^*$ (Lemma~\ref{lemma:bound-trajectories-QG}) and the completeness of $\calM$.
    The Hessian of $f$ at $x^*$ is positive definite by Lemma~\ref{lem:MB}.
    Thus, Theorem~\ref{thm:globalsquaredistsinglecp} applies.
\end{proof}

We are ready to proceed with the technical lemmas.
The first one is essentially the (local) Morse lemma, only globalized so that we get a diffeomorphism from all of $\calM$ to itself.
Recall $\dist$ is the Riemannian distance on $\calM$.

\begin{lemma} \label{lem:localsquaredistsinglecp}
    Let $f \colon \calM \to \reals$ be smooth.
    Assume $x^*$ is a critical point of $f$ and that the Hessian of $f$ at $x^*$ is positive definite.
    Then there exist a positive radius $r > 0$ and a diffeomorphism $\psi \colon \calM \to \calM$ such that
    $\psi(x^*) = x^*$ and 
    \begin{align*}
        (f \circ \psi)(x) = f(x^*) + \dist(x, x^*)^2 && \textrm{ for all } && x \in \calM && \textrm{ such that } && \dist(x, x^*) \leq r.
    \end{align*}
\end{lemma}
\begin{proof}
  This is a simple consequence of the Morse lemma (which provides a \emph{local} diffeomorphism that makes $f$ into a squared distance function) and of the Palais--Cerf theorem (which extends that local diffeomorphism into a global one).
  Details are in Appendix~\ref{app:extension-morse}.
\end{proof}

The next lemma is a simple helper, in keeping with standard arguments as seen for example in~\citep[Thm.~3.1]{milnor1963morse}.
We use it to prove the more involved Lemma~\ref{lem:globalquadraticsinglecp}.

\begin{lemma}[Rescaled gradient flow] \label{lem:GFrescaledoncoerciveinvexM}
    Let $f \colon \calM \to \reals$ be smooth and coercive.
    Assume $f$ has a unique critical point $x^*$, and that $f(x^*) = 0$.
    For $x \neq x^*$, let $t \mapsto \nu(x, t)$ denote the solution of the rescaled gradient flow
    \begin{align}
        \ddt \nu(x, t) = \frac{1}{\|\grad f(\nu(x, t))\|_{\nu(x, t)}^2} \grad f(\nu(x, t)) && \textrm{ with } && \nu(x, 0) = x.
        \label{eq:rescaledflow}
    \end{align}
    Then, $\nu(x, t)$ is smoothly defined for all $x \neq x^*$ and $t \in (-f(x), \infty)$.
    Moreover, $f(\nu(x, t)) = f(x) + t$ so that $\nu(x, t) \to x^*$ for $t \to -f(x)$.
\end{lemma}
\begin{proof}
    This is a consequence of the fundamental theorem of flows together with the fact that $\ddt f(\nu(x, t)) = \inner{\grad f(\nu(x, t))}{\ddt \nu(x, t)}_{\nu(x, t)} = 1$, by design.
    See Appendix~\ref{app:rescaledflow} for details.
\end{proof}

The heavy lifting is done by the next lemma.
This is where each gradient flow trajectory on $\calM$ is mapped to a ray of $\Rn$ (amounting to a diffeomorphism $\varphi$ from $\calM$ to $\Rn$) such that $(f \circ \varphi^{-1})(y) = f(x^*) + \|y\|^2$.
In particular, the level sets of $f$ are deformed by $\varphi$ to become spheres.

\begin{lemma} \label{lem:globalquadraticsinglecp}
    Let $f \colon \calM \to \reals$ be a smooth, coercive function, and suppose $f$ has a unique critical point $x^*$.
    Assume further that there exists $r > 0$ such that $f(x) = f(x^*) + \dist(x, x^*)^2$ for all $x \in \calM$ with $\dist(x, x^*) \leq r$.
    Then, there exists a diffeomorphism $\varphi \colon \calM \to \Rn$ such that $f(x) = f(x^*) + \|\varphi(x)\|^2$ for all $x \in \calM$.
\end{lemma}
\begin{proof}
    Without loss of generality, we may assume $f(x^*) = 0$.
    By Lemma~\ref{lem:coerciveuniquecpglobalmin}, we know $f(x) \geq f(x^*) = 0$ for all $x$.
    By Lemma~\ref{lem:GFrescaledoncoerciveinvexM}, the flow map $\nu$~\eqref{eq:rescaledflow} is smoothly defined on
    \begin{align*}
        \big\{ (x, t) \in \calM \times \reals : x \neq x^* \textrm{ and } t > -f(x) \big\},
    \end{align*}
    in such a way that $\nu(x, 0) = x$ and $f(\nu(x, t)) = f(x) + t$.

    We aim to map each $x \in \calM$ to a vector in $\Rn$.
    Since $\T_{x^*}\calM$ is isometric to $\Rn$, it is enough to map each $x$ to a vector in $\T_{x^*}\calM$.
    (Explicitly, these can then be expanded in an orthonormal basis of $\T_{x^*}\calM$ to obtain a vector of coordinates in $\Rn$ with the same norm.)

    To do so, reduce $r$ if need be so it becomes smaller than the injectivity radius of $\calM$ at $x^*$ (but still positive). 
    Then, by definition, the Riemannian exponential $\Exp_{x^*}$ is a diffeomorphism from the open ball of radius $r$ around the origin in $\T_{x^*}\calM$ to the open ball of radius $r$ around $x^*$ on $\calM$.
    The inverse of $\Exp_{x^*}$ on those domains is denoted by $\Log_{x^*}$.
    On these domains, we have $\dist\!\big(\Exp_{x^*}(v), x^*\big) = \|v\|_{x^*}$ and $\|\Log_{x^*}(x)\|_{x^*} = \dist(x, x^*)$~\citep[Prop.~6.11]{lee2018riemannian}, \citep[Prop.~10.22]{boumal2020intromanifolds}.

    We separate $\calM$ in two overlapping regions, and define a mapping $\varphi \colon \calM \to \T_{x^*}\calM$ on each region.
    To this end, consider an arbitrary $x \in \calM$.
    \begin{itemize}
    \item On the one hand, if $\dist(x, x^*) < r$, then we can use the fact that $f(x) = \dist(x, x^*)^2$ to define a vector in $\T_{x^*}\calM$ as follows:
        \begin{align*}
            \varphi(x) = \Log_{x^*}(x) && \textrm{ for } x \in \calM \textrm{ such that } \dist(x, x^*) < r.
        \end{align*}
        Thus we have $\|\varphi(x)\|_{x^*}^2 = \dist(x, x^*)^2 = f(x)$, as desired.

    \item On the other hand, if $\dist(x, x^*) > r/2$, then we can use the flow $\nu$~\eqref{eq:rescaledflow} to bring $x$ to a point $x' = \nu(x, t)$ such that $\dist(x', x^*) = r/2$.
        Such a time $t > -f(x)$ exists because the trajectory $s \mapsto \nu(x, s)$ converges to $x^*$ as $s \to -f(x)$, so it must traverse the sphere of radius $r/2$ at least once.
        Then, we know two things:
        \begin{align*}
        \textrm{(a)} \quad f(x') = \dist(x', x^*)^2 = (r/2)^2 = r^2/4
        && \textrm{ and } &&
        \textrm{(b)} \quad f(\nu(x, t)) = f(x) + t.
        \end{align*}
        Thus, $t$ is actually unique: $t = r^2/4 - f(x)$.
        We can then define a vector in $\T_{x^*}\calM$ as follows:
        \begin{align*}
            \varphi(x) = \frac{2\sqrt{f(x)}}{r} \Log_{x^*}\!\Big(\nu\big(x, r^2/4 - f(x)\big)\Big) && \textrm{ for } x \in \calM \textrm{ such that } \dist(x, x^*) > r/2.
        \end{align*}
        Here too, $\|\varphi(x)\|_{x^*}^2 = \frac{4f(x)}{r^2} \dist(x', x^*)^2 = f(x)$, as desired.
    \end{itemize}
    It is clear that the two separate definitions of $\varphi$ are smooth.
    Two tasks remain:
    \begin{enumerate}
        \item Show that the two definitions agree on the overlap of the two regions: $\{x \in \calM : r/2 < \dist(x, x^*) < r\}$, upon which we can claim that $\varphi$ is smooth on all of $\calM$.
        \item Argue that $\varphi$ has a smooth inverse, to confirm that $\varphi$ is a diffeomorphism.
    \end{enumerate}

    \paragraph{\boldmath Definitions of $\varphi$ agree on overlap:} Regarding the first item, consider a point $x \in \calM$ such that $r/2 < \dist(x, x^*) < r$.
    Then $\nu(x, t)$ remains in the ball of radius $r$ around $x^*$ for all $t \in (-f(x), 0]$.
    In that ball, $f(\cdot) = \dist(\cdot, x^*)^2$.
    Thus, we can integrate the flow and write the solution $t \mapsto \nu(x, t)$ explicitly: it follows the geodesic from $x$ down to $x^*$ as $\nu(x, t) = \Exp_{x^*}\big(\alpha(t) \Log_{x^*}(x)\big)$ for some scalar function $\alpha$.
    Moreover, $\alpha$ satisfies
    \begin{equation*}
        \alpha(t)^2 f(x) = \alpha(t)^2 \dist(x, x^*)^2 = \dist\!\big(\nu(x, t), x^*\big)^2 = f(\nu(x, t)) = f(x) + t,
    \end{equation*}
    meaning that $\alpha(t) = \sqrt{\frac{f(x) + t}{f(x)}}$.
    At $t = r^2/4 - f(x)$, we find $\alpha(t) = \frac{r}{2\sqrt{f(x)}}$, and so
    \begin{equation*}
        \Log_{x^*}\!\big(\nu(x, t)\big) = \alpha(t) \Log_{x^*}(x) = \frac{r}{2\sqrt{f(x)}} \Log_{x^*}(x).
    \end{equation*}
    This confirms that $\varphi(x) = \Log_{x^*}(x)$ with both the first and second definitions of $\varphi$.
    Thus, $\varphi \colon \calM \to \T_{x^*}\calM$ is well defined and smooth.

    \paragraph{\boldmath Smooth inverse of $\varphi$:} Now turning to the second item, we need to show that $\varphi$ is a diffeomorphism.
    To this end, we build its inverse and show that it is smooth too.
    Consider $\xi \colon \T_{x^*}\calM \to \calM$, defined as follows:
    \begin{align*}
        \xi(v) = \begin{cases}
            \Exp_{x^*}(v) & \textrm{ if } \|v\|_{x^*} < r, \\
            \nu(x, t) & \textrm{ if } \|v\|_{x^*} > r/2, \textrm{ where } x = \Exp_{x^*}\!\left( \frac{r/2}{\|v\|_{x^*}} v \right) \textrm{ and } t = \|v\|_{x^*}^2 - f(x).
        \end{cases}
    \end{align*}

    Here too, $\xi$ is smoothly defined on two overlapping domains, and we need to check that the two definitions agree on their intersection.
    To see this, consider a point $v \in \T_{x^*}\calM$ such that $r/2 < \|v\|_{x^*} < r$.
    Then $x = \Exp_{x^*}\!\Big( \frac{r/2}{\|v\|_{x^*}} v \Big)$ is in the ball of radius $r$ around $x^*$.
    Moreover, the flow $\nu(x, t)$ remains in that ball for all $t \in (-f(x), r^2 - f(x))$ as then $f(\nu(x, t)) \leq r^2$.
    Thus, in that time interval, we can integrate the flow and write the solution $t \mapsto \nu(x, t)$ explicitly as we did before: $\nu(x, t) = \Exp_{x^*}(\alpha(t) v)$ for some function $\alpha$, which satisfies
    \begin{align*}
        \alpha(t)^2 \|v\|_{x^*}^2 = \dist\!\big(\nu(x, t), x^*\big)^2 = f(\nu(x, t)) = f(x) + t.
    \end{align*}
    It follows that $\alpha(t) = \frac{\sqrt{f(x) + t}}{\|v\|_{x^*}}$.
    At $t = \|v\|_{x^*}^2 - f(x)$, we find $\alpha(t) = 1$ and $\nu(x, t) = \Exp_{x^*}(v)$.
    This confirms that $\xi(v)$ is equal to $\Exp_{x^*}(v)$ with both the first and second definitions of $\xi$.

    It remains to check that $\varphi$ and $\xi$ are inverses of each other.
    For $v \in \T_{x^*}\calM$ such that $\|v\|_{x^*} < r$, we have
    \begin{align*}
      \varphi(\xi(v)) = \varphi\big(\Exp_{x^*}(v)\big) = \Log_{x^*}\!\big(\Exp_{x^*}(v)\big) = v.
    \end{align*}
    Now let $v \in \T_{x^*}\calM$ such that $\|v\|_{x^*} > r/2$.
    In this case, $\xi(v) = \nu(x, t)$ with $x := \Exp_{x^*}\!\Big( \frac{r/2}{\|v\|_{x^*}} v \Big)$ and $t := \|v\|_{x^*}^2 - f(x)$.
    Using the identities $f(\nu(x, t)) = f(x) + t$ and $\nu(\nu(x, t_1), t_2) = \nu(x, t_1 + t_2)$, we find
    \begin{align*}
        \varphi(\xi(v)) & = \varphi(\nu(x, t)) \\
                        & = \frac{2\sqrt{f(\nu(x, t))}}{r} \Log_{x^*}\!\Big(\nu\Big(\nu(x, t), r^2/4 - f(\nu(x, t))\Big)\Big) \\
                        & = \frac{2\sqrt{f(x) + t}}{r} \Log_{x^*}\!\Big(\nu\big(x, r^2/4 - f(x)\big)\Big) \\
                        & = \frac{2\|v\|_{x^*}}{r} \Log_{x^*}(x) \\
                        & = v,
    \end{align*}
    where we also used that
    \begin{align*}
      f(x) = \dist(x, x^*)^2 = r^2/4 && \text{so that} && \nu\big(x, r^2/4 - f(x)\big) = \nu(x, 0) = x.
    \end{align*}
    In all cases, $\varphi \circ \xi$ is the identity on $\T_{x^*}\calM$.

    The other way around, we now let $x \in \calM$ and show that $\xi(\varphi(x)) = x$.
    If $\dist(x, x^*) < r$, then
    \begin{align*}
      \xi(\varphi(x)) = \xi(\Log_{x^*}(x)) = \Exp_{x^*}(\Log_{x^*}(x)) = x.
    \end{align*}
    If $\dist(x, x^*) > r/2$, then $v := \varphi(x)$ has norm $\|v\|_{x^*} = \sqrt{f(x)} > r/2$.
    This is because the value of $f$ along a trajectory increases from $0$ to $r^2/4$ as it travels from $x^*$ to the sphere of radius $r/2$; and then it keeps increasing so that the trajectory cannot go back into the sphere of radius $r/2$.
    Thus,
    \begin{align*}
        \xi(\varphi(x)) = \xi(v) = \nu\!\left( x', f(x) - f(x') \right) && \textrm{ with } && x' = \Exp_{x^*}\!\bigg( \frac{r}{2\sqrt{f(x)}} v \bigg).
    \end{align*}
    Plugging the expression for $v = \varphi(x)$ into the definition of $x'$, we find $x' = \nu\big(x, r^2/4 - f(x)\big)$.
    Here too using the property $\nu(\nu(x, t_1), t_2) = \nu(x, t_1 + t_2)$, it follows that
    \begin{align*}
        \xi(\varphi(x)) & = \nu\!\left( \nu\big(x, r^2/4 - f(x)\big), f(x) - f(x') \right) = \nu\big(x, r^2/4 - f(x')\big) = \nu(x, 0) = x,
    \end{align*}
    because $f(x') = \dist(x', x^*)^2 = r^2/4$.
    In all cases, $\xi \circ \varphi$ is the identity on $\calM$.
\end{proof}

\section{The general case} \label{sec:generalcase}

In this section, we prove Theorem~\ref{thm:plnonlinlstsq} about globally \pl{} functions $f \colon \calM \to \reals$, following the strategy laid out in Section~\ref{sec:prooftechnique}.
We start by defining and studying the end-point map $\pi$ in Section~\ref{sec:endpointmap}.
In particular, we conclude there that the set of minimizers $S$ is a smooth manifold (a strong deformation retract of $\calM$) and that $\pi \colon \calM \to S$ is a smooth submersion.
Then, we proceed in Section~\ref{sec:fiberbundle} to show that $\pi$ is a smooth fiber bundle---a trivial one under the assumption that $\calM$ is contractible.
The construction is explicit so as to exert additional control over $f$.
The proof of Theorem~\ref{thm:plnonlinlstsq} then reduces to a corollary, with details in Section~\ref{sec:proofplnonlinlstsq}.

\subsection{The end-point map of negative gradient flow} \label{sec:endpointmap}

Let us open with a few basic facts about negative gradient flow on $f$.

\begin{lemma} \label{lem:basicflowproperties}
    Let $f \colon \calM \to \reals$ be smooth and globally \pl{}.
    Negative gradient flow on $f$ defines a flow map $\Phi \colon (y, t) \mapsto \Phi^t(y)$ via
    \begin{align*}
        \ddt \Phi^t(y) = -\grad f(\Phi^t(y)) && \textrm{ and } && \Phi^0(y) = y.
    \end{align*}
    The following properties hold:
    \begin{enumerate}
        \item The domain of $\Phi$ is open in $\calM \times \reals$, and $\Phi$ is smooth. 
        \item For all $y \in \calM$, the trajectory $t \mapsto \Phi^t(y)$ is defined for all $t \geq 0$.
        \item For all $y \in \calM$, the limit $\Phi^{\infty}(y) := \lim_{t \to \infty} \Phi^t(y)$ exists and is a critical point of $f$.
        \item For all $t \in \reals$, the map $\Phi^t$ is a diffeomorphism from $M_t$ to $M_{-t}$, where $M_t = \{y \in \calM : (y, t) \textrm{ is in the domain of } \Phi\}$. In particular, $M_t = \calM$ for all $t \geq 0$.
    \end{enumerate}
\end{lemma}
\begin{proof}
    See the fundamental theorem of flows in~\citep[Thm.~9.12]{lee2012smoothmanifolds}.
    The fact that all trajectories are defined for all positive times follows from the Escape Lemma~\citep[Lem.~9.19]{lee2012smoothmanifolds} and the boundedness of their length for $t \in [0, \infty]$, owing to the \pl{} condition (Lemma~\ref{lemma:bound-trajectories-QG}).
    The latter further implies that they converge to a point, which must be critical.
\end{proof}

Since each trajectory of negative gradient flow on $f$ has a well-defined limit, we can define the  \emph{end-point map}
\begin{align}
    \pi \colon \calM \to S, && \pi(y) := \lim_{t \to \infty} \Phi^t(y).
    \label{eq:pi}
\end{align}
We know $\pi$ is surjective since it is identity on $S$.
Using standard arguments, we further argue in Proposition~\ref{prop:picontinuous} that $\pi$ is continuous and moreover that $\calM$ strongly deformation retracts to $S$ (Definition~\ref{def:deformationretract}).
This implies that $\calM$ and $S$ are \emph{homotopy equivalent}~\citep[p.~200]{lee2011topological}.
Therefore, $\calM$ and $S$ share topological properties called \emph{homotopy invariants}, including contractibility~\citep[Ex.~7.41]{lee2011topological}.

The construction of deformation retractions based on gradient flows is classical, with similar examples in~\citep[Thm.~5]{lojasiewicz1963propriete} and~\citep[Prop.~3]{kurdyka1998gradientsomin} applied to other classes of functions (real-analytic and definable in an o-minimal structure, respectively).

\begin{proposition} \label{prop:picontinuous}
    Let $f \colon \calM \to \reals$ be a smooth, globally $\mu$-\pl{} function, and let $S$ denote its set of critical points.
    Then the end-point map $\pi$~\eqref{eq:pi} is continuous, and $\pi(x) = x$ if and only if $x$ is in $S$.
    In particular, $S$ is connected.
    Moreover, $\calM$ strongly deformation retracts to $S$ so that $\calM$ is contractible if and only if $S$ is so.
\end{proposition}
\begin{proof}
    Recall from Lemma~\ref{lemma:bound-trajectories-QG} that $S$ is non-empty and closed.
    By Lemma~\ref{lem:basicflowproperties}, the map $\pi = \Phi^\infty$ is well defined, and the flow map $\Phi$ is continuous on its open domain in $\calM \times \reals$, which contains $\{(y, t) : t \geq 0\}$.

    \paragraph{\boldmath The map $\pi$ is continuous:}
    The proof consists in observing that $\pi$ is continuous on a neighborhood of $S$, and then globalizing via the identity $\pi = \pi \circ \Phi^t$ with some large $t$.
    Explicitly, fix $y \in \calM$ and let $x = \pi(y)$.
    Pick an arbitrary neighborhood $U$ of $x$.
    It is enough to build a neighborhood $B$ of $y$ such that $\pi(B) \subseteq U$.
    To this end, let $U'$ be a smaller neighborhood of $x$ whose closure is in $U$.
    Since trajectories have bounded length by the \pl{} condition (Lemma~\ref{lemma:bound-trajectories-QG}), there exists an open neighborhood $V$ of $x$ such that if $z$ is in $V$ then $\Phi^s(z)$ is in $U'$ for all $s \geq 0$.
    In particular, $\pi(V)$ is in the closure of $U'$, hence it is in $U$.
    Select $t \geq 0$ such that $z := \Phi^t(y)$ is in $V$.
    Let $W$ be the intersection of $V$ with the domain of $\Phi^{-t}$ (it contains $z$ and is open since the domain of $\Phi$ is open): this is a neighborhood of $z$.
    Define $B = \Phi^{-t}(W)$.
    Then $B$ is a neighborhood of $y$ (because $\Phi^{t}$ is continuous), and $\pi(B) = \pi(\Phi^t(B)) = \pi(W) \subseteq \pi(V) \subseteq U$, as needed.

    By assumption, $\calM$ is connected.
    We just argued $S$ is a continuous image of $\calM$, as $S = \pi(\calM)$.
    Thus, $S$ is connected.

    \paragraph{Deformation retraction:}
    Define the reparameterization $t(s) = s/(1-s)$, which is strictly increasing and maps $[0, 1]$ to $[0, \infty]$.
    Consider the map $F \colon \calM \times [0, 1] \to \calM$ defined as
    \begin{align*}
      F(y, s) = \Phi^{t(s)}(y).
    \end{align*}
    By the above properties, $F$ is well defined.

    From the continuity of $\pi$ we deduce that $F$ is continuous.
    Also, for all $y \in \calM$ we have
    \begin{align*}
      F(y, 0) = \Phi^0(y) = y && \textrm{ and } && F(y, 1) = \Phi^{\infty}(y) = \pi(y) \in S.
    \end{align*}
    Additionally, if $y$ is in $S$ then $\Phi^t(y) = y$ for all $t \geq 0$ (points in $S$ are fixed points of gradient flow), and hence $F(y, s) = \Phi^{t(s)}(y) = y$ for all $y \in S$ and $s \in [0, 1]$.
    We conclude that $F$ is a strong deformation retraction of $\calM$ onto $S$ (Definition~\ref{def:deformationretract}).

    \paragraph{Contractibility:}
    From the previous paragraph, it is immediate that $S$ and $\calM$ share various topological properties: they are homotopy equivalent. 
    This holds in particular for contractibility (Definition~\ref{def:contractible}).
    Let us spell out the details.

    Let $x$ be a point in $S$.

    If $\calM$ is contractible, then it deformation retracts onto any of its points, and in particular onto $\{x\}$. 
    Let $G \colon \calM \times [0, 1] \to \calM$ be a deformation retraction of $\calM$ onto $\{x\}$.
    (For example, when $\calM = \Rn$, one can take $G(y, t) = (1-t) y + t x$.)
    Then, $S$ also deformation retracts onto $\{x\}$ via $H \colon S \times [0, 1] \to S$ defined by $H(y, t) = \pi(G(y, t))$ (using both that $\pi$ is continuous and that it is identity on $S$).
    Therefore, $S$ is also contractible.

    The other way around, assume $S$ is contractible and let $H \colon S \times [0, 1] \to S$ deformation retract $S$ onto $\{x\}$.
    Using $F$ as defined above, build the map $G \colon \calM \times [0, 1] \to \calM$ as follows:
    \begin{align*}
        G(y, t) = \begin{cases}
            F(y, 2t) & \textrm{ if } t \in [0, 1/2], \\
            H(2t-1, \pi(y)) & \textrm{ if } t \in [1/2, 1].
        \end{cases}
    \end{align*}
    This map is continuous.
    It deformation retracts $\calM$ onto $\{x\}$, hence $\calM$ is contractible.
\end{proof}

Using more sophisticated tools, one can further show that $\pi$ is smooth, and even that it is a smooth submersion.
The heavy lifting is done by~\citet{falconer1983limitmapping}, whose proof relies on the center stable manifold theorem~\citep[Thm.~5.1]{hirsch1977invariant}.
Before we can apply those tools, we need to make sure $S$ is a smooth manifold.
For this part, we use the recent results integrated in Lemma~\ref{lem:MB}.

\begin{proposition} \label{prop:Ssmoothpisubmersion}
    (Continued from Proposition~\ref{prop:picontinuous}.)
    The set $S$ of critical points of $f$ is a smooth, properly embedded submanifold of $\calM$, and $\pi \colon \calM \to S$ is a smooth submersion.
\end{proposition}
\begin{proof}
    From Proposition~\ref{prop:picontinuous}, $S$ is connected.
    Combining with Lemma~\ref{lem:MB}, we deduce that $S$ is a smooth manifold embedded in $\calM$.
    It is properly embedded because $S$ is a closed subset of $\calM$~\citep[Prop.~5.5]{lee2012smoothmanifolds}.

    \paragraph{$\pi$ is smooth:}
    This is not trivial: it follows from~\citep[Thm.~5.1]{falconer1983limitmapping}.
    Let us add some context as to why.

    Falconer's theorem applies to the limit-point map of a \emph{discrete} dynamical system $y_{t+1} = g(y_t)$ with some $g \colon \calM \to \calM$.
    For our case, we can take $g = \Phi^1$, that is, the time-one map of negative gradient flow, which is smooth by Lemma~\ref{lem:basicflowproperties}.
    It is clear that the set of fixed points of $g$ is exactly $S$. 
    Pick one of these fixed points, $x \in S$.
    It is known that $\D g(x) = e^{-\nabla^2 f(x)}$ (exponential of negative the Hessian of $f$ at $x$)---this is a particular case of a standard fact which can be derived from~\citep[\S32.6, Lem.~8]{arnold2006ode} (details in~\citep{rttb2025timeonemap} or~\citep[Lem.~4.19]{banyaga2004morsehomology}).

    Recall from Lemma~\ref{lem:MB} that $\nabla^2 f(x)$ splits $\T_x \calM$ in two orthogonal subspaces which correspond to the tangent space $\T_x S$ and the normal space $\N_x S$ of
    $S$ at $x$.
    Indeed, $\T_x S$ is the kernel of $\nabla^2 f(x)$ because $\nabla f$ is constant (zero) on $S$.
    The orthogonal complement is also an invariant subspace of $\nabla^2 f(x)$: it corresponds to the nonzero eigenvalues of $\nabla^2 f(x)$, which are all at least $\mu$.
    Therefore, $\D g(x) \colon \T_x S \to \T_x S$ is the identity map, while $\D g(x) \colon \N_x S \to \N_x S$ is a symmetric map with all of its eigenvalues in the interval $(0, e^{-\mu}]$.
    In particular, $\D g(x)$ is a strict contraction on $\N_x S$.
    These considerations imply that $S$ is \emph{pseudo-hyperbolic} for $g$, as per the definition in~\citep[\S5]{falconer1983limitmapping}.
    Therefore, we may apply~\citep[Thm.~5.1]{falconer1983limitmapping} and conclude that $\pi = \Phi^\infty = g^\infty$ is indeed smooth.

    \paragraph{$\pi$ is a submersion:}
    To show that $\pi \colon \calM \to S$ is a smooth submersion, we argue that $\D\pi(y) \colon \T_y\calM \to \T_{\pi(y)} S$ is surjective for all $y \in \calM$.
    To this end, first fix an arbitrary $x \in S$.
    For any $u \in \T_x S$, let $c$ be a smooth curve on $S$ such that $c(0) = x$ and $c'(0) = u$.
    Then, $\pi(c(t)) = c(t)$ for all $t$, so that (after differentiating and evaluating at $t = 0$) we find $\D\pi(x)[u] = u$, that is, $\D\pi(x)$ is identity on $\T_x S$.
    It follows that $\D\pi(x)$ is surjective for all $x \in S$.
    By continuity, $\D\pi(y)$ is surjective for all $y$ in a neighborhood $U$ of $S$.
    Now take $y \in \calM$ arbitrary.
    Since $\pi(y)$ is in $S$, there exists $t \geq 0$ such that $\Phi^t(y)$ is in $U$.
    Moreover, $\pi = \pi \circ \Phi^t$.
    Differentiating the latter at $y$, we find
    \begin{align*}
        \D\pi(y) = \D\pi(\Phi^t(y)) \circ \D\Phi^t(y).
    \end{align*}
    By design, $\Phi^t(y)$ is in $U$, hence $\D\pi(\Phi^t(y))$ is surjective.
    By Lemma~\ref{lem:basicflowproperties}, $\Phi^t$ is a diffeomorphism from $\calM$ to its image, hence $\D\Phi^t(y)$ is invertible.
    It follows that $\D\pi(y)$ is surjective for all $y$, that is, $\pi$ is a submersion.
\end{proof}

The fiber of $\pi$~\eqref{eq:pi} for a critical point $x \in S$~\eqref{eq:S} is the set
\begin{align*}
    \calF = \pi^{-1}(x) = \{ y \in \calM : \pi(y) = x \}.
\end{align*}
It contains all initial points $y$ from where negative gradient flow on $f$ converges to $x$.
These fibers are nice manifolds themselves, and restricting $f$ to a fiber retains the \pl{} property.

\begin{proposition} \label{prop:plrestrictedtofiber}
    Let $f \colon \calM \to \reals$ be smooth and globally $\mu$-\pl{}.
    If $x$ is a critical point of $f$, then the fiber $\calF = \pi^{-1}(x)$ is a contractible, properly embedded smooth submanifold of $\calM$.
    With the Riemannian submanifold structure on $\calF$, the restriction $f|_{\calF} \colon \calF \to \reals$ is smooth and globally $\mu$-\pl{} with $x$ as its unique critical point.
    In particular, $\calF$ is diffeomorphic to $\Rk$ with $k = \dim \calM - \dim S$.
\end{proposition}
\begin{proof}
    We know $\pi \colon \calM \to S$ is a smooth submersion by Proposition~\ref{prop:Ssmoothpisubmersion}.
    Each fiber is a level set of $\pi$, hence it is a properly embedded smooth submanifold of $\calM$~\citep[Cor.~5.13]{lee2012smoothmanifolds}.
    It is also clear that $\calF$ is contractible: simply flow each point to $x$ using the negative gradient flow of $f$ (explicitly, consider the map $F$ in the proof of Proposition~\ref{prop:picontinuous}, restricted to $\calF \times [0, 1] \to \calF$).

    Endow $\calF$ with the Riemannian submanifold structure inherited from $\calM$.
    Then, it is complete because it is properly embedded in $\calM$ which is itself complete~\citep[13-18(b)]{lee2012smoothmanifolds}.

    Observe that the restriction of $f$ to $\calF$, denoted here by $g = f|_{\calF}$, is itself smooth and globally $\mu$-\pl{}, with a single critical point at $x$.
    Indeed, the trajectories of negative gradient flow for $f$ initialized in $\calF$ remain in $\calF$ by definition, so that $\grad f(y)$ is tangent to $\calF$ for all $y \in \calF$.
    The gradient of $g$ at $y$ is the orthogonal projection of $\grad f(y)$ to $\T_y \calF$, but it is already tangent hence $\grad g(y) = \grad f(y)$~\citep[eq.~(3.37)]{AMS08}.
    In particular, the norms of $\grad g(y)$ and $\grad f(y)$ are equal.
    By definition of the \eqref{eq:PL} property, it is now clear that $g$ is $\mu$-\pl{} simply because $f$ has that quality and $x$ is a global minimizer of $f$ hence also of $g$.
    The set of critical points of $g$ is $\calF \cap S = \{x\}$, as claimed.

    Apply Theorem~\ref{thm:globalsquaredistsinglecppl} to deduce that $\calF$ is diffeomorphic to $\Rk$ with $k = \dim \calF$.
\end{proof}

At this point, we can already claim that negative gradient flow on $f$ (without any particular assumption on $S$) induces a smooth fiber bundle structure (although it may or may not be trivial)---see Definition~\ref{def:smoothfiberbundle}.
This claim relies on a strong result by~\citet{meigniez2002submersions}.
In the next section, we give an explicit proof that also provides a \emph{trivial} fiber bundle structure assuming contractibility.

\begin{corollary} \label{cor:fiberbundlegeneral}
    Let $f \colon \calM \to \reals$ be smooth and globally \pl{}.
    Then, $\pi \colon \calM \to S$ is a smooth fiber bundle with fibers diffeomorphic to $\Rk$, $k = \dim \calM - \dim S$.
\end{corollary}
\begin{proof}
    From Proposition~\ref{prop:Ssmoothpisubmersion}, we know $S$ is a smooth submanifold of $\calM$ and $\pi$ is a surjective smooth submersion.
    Each fiber of $\pi$ is diffeomorphic to $\Rk$ by Proposition~\ref{prop:plrestrictedtofiber}.
    The claim now follows from \citep[Cor.~31]{meigniez2002submersions} which states that if the fibers of a surjective smooth submersion are diffeomorphic to $\Rk$ then that submersion is a smooth fiber bundle.
\end{proof}

\subsection{The fiber bundle structure} \label{sec:fiberbundle}

\begin{figure}[t]
    \centering
    \includegraphics[width=1.0\textwidth]{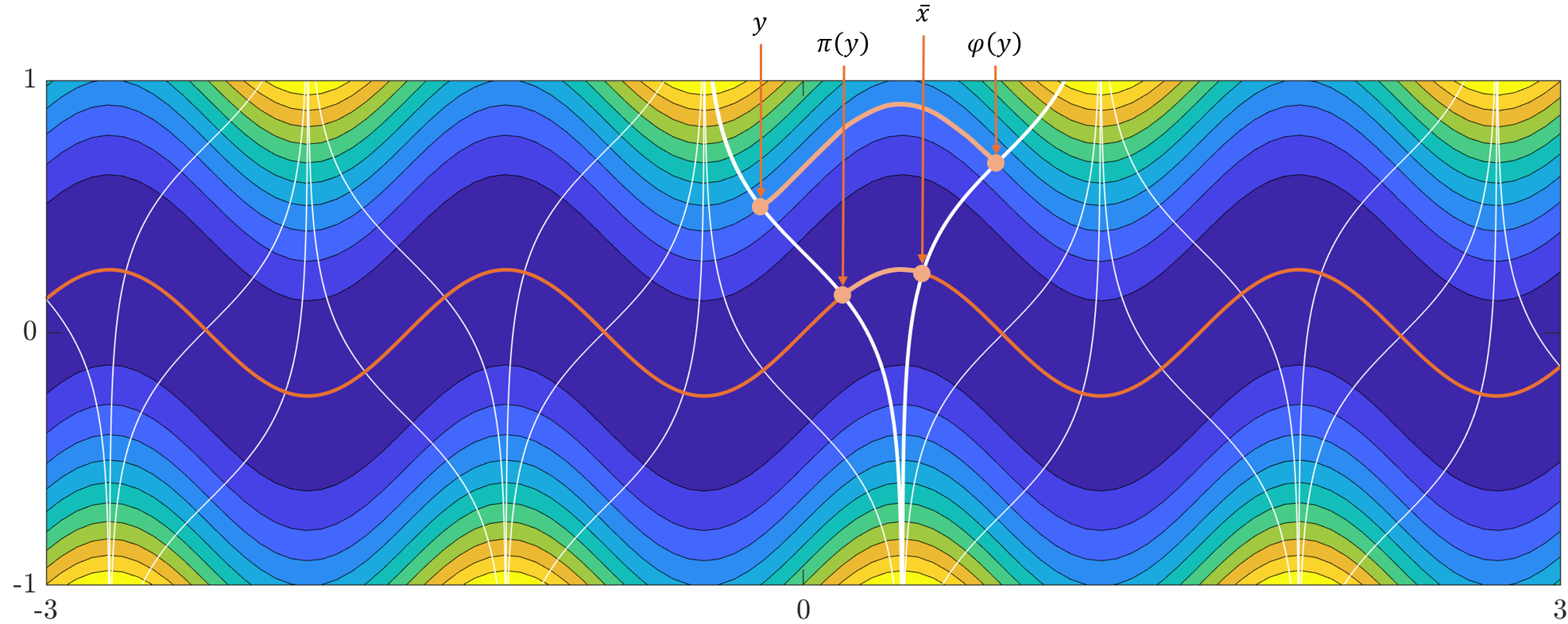}
    \caption{Illustration of the proof for Theorem~\ref{thm:pitrivial}. Background colors indicate level sets of $f(x) = \frac{1}{2}\big(\frac{1}{4}\sin(4x_1) - x_2\big)^2$, which is globally \pl{} on $\reals^2$. The set of minimizers $S$ is the dark orange curve. Fix $\bar{x} \in S$: the white curve passing through it is the fiber $\pi^{-1}(\bar x)$. Choose $y \in \reals^2$; then, $\pi(y)$ is on $S$ and the white curve passing through it is its fiber $\pi^{-1}(\pi(y))$. The points $\pi(y)$ and $\bar x$ can be connected by a smooth curve $c$ on $S$, with $c(0) = \pi(y)$ and $c(1) = \bar{x}$. The proof builds a lifted curve $\gamma$ in $\reals^2$ such that $\gamma(0) = y$, $\pi \circ \gamma = c$ and $f$ remains constant along $\gamma$. In particular, $\gamma(1)$ is on the fiber of $\bar{x}$. We call that point $\varphi(y)$, and we have $f(\varphi(y)) = f(y)$. This is done in a way that $\varphi$ is smooth.}
    \label{fig:PLflowbundle}
\end{figure}

It is only now that we introduce the assumption that $S$ is contractible (Definition~\ref{def:contractible}).
At a high level, since we know from Corollary~\ref{cor:fiberbundlegeneral} that the end-point map $\pi \colon \calM \to S$~\eqref{eq:pi} is a fiber bundle, it is clear from general results in differential topology that $\pi$ is a trivial fiber bundle if the base space $S$ is contractible: see~\citep[\S3.4B]{abraham1988manifolds} (including the note about \emph{smooth} fiber bundles at the end), or the \emph{covering homotopy theorem}~\citep[\S4, Thm.~1.5]{hirsch1976differential} (including Ex.~2, 3 thereafter) in the context of smooth vector bundles.

Here, we do not rely on those results, nor do we use Corollary~\ref{cor:fiberbundlegeneral}.
Instead, we build explicit trivializations of $\pi$ which allow us to retain control over the value of $f$.
This later enables the claim about the quadratic nature of $f$, and makes for a more transparent proof.\footnote{Before resorting to a bespoke proof of the fiber bundle structure, we were hoping to rely more on existing literature (see Section~\ref{sec:relatedwork}). Unfortunately, all results we could find involve compactness assumptions that (in our setting) would force $S$ to be a singleton (Corollary~\ref{cor:compactSsingleton}). The added benefit of crafting our own trivialization maps is that we can build them in a way that they play nicely with $f$.}

The construction below relies only on (a) the propositions from the previous section; (b) other basic properties of \pl{} functions; and (c) standard results for ordinary differential equations and smooth manifolds.

In spirit, it is similar to how one might prove \emph{Ehresmann's fibration theorem}.
The latter states that a \emph{proper} surjective smooth submersion is a (not necessarily trivial) smooth fiber bundle.
In our case, $\pi$ is typically \emph{not} proper because its fibers are diffeomorphic to $\Rk$.
Upon closer inspection, properness is used there to ensure that curves on $S$ can be lifted entirely to (special) curves on $\calM$.
These curves are solutions to ODEs: they exist for as long as they do not escape to infinity.
Such escapes are ruled out by properness, so the curves exist for all times.
In our case, we can ensure the same via the \pl{} assumption on $f$, by tapping into the relation between $\pi$ and $f$.

Figure~\ref{fig:PLflowbundle} illustrates the curve lifting part of our proof.
It goes as follows.
Fix $\bar x \in S$.
We let $y \in \calM$ be arbitrary, and push it to $S$ as $x = \pi(y)$.
Contractibility allows us to connect $x$ to $\bar x$ with a smooth curve $c \colon [0, 1] \to S$, $c(0) = x, c(1) = \bar x$, in a way that the curve itself depends smoothly on $x$.
We want to \emph{lift} $c$ to a curve $\gamma \colon [0, 1] \to \calM$ such that $\gamma(0) = y$.
That is, we aim to have $\pi \circ \gamma = c$.
Differentiating this readily shows that
\begin{align*}
    \D\pi(\gamma(t))[\gamma'(t)] = c'(t).
\end{align*}
This is not enough to determine $\gamma$, because $\D\pi$ has a kernel.
So, in addition, we require that $\gamma$ should be orthogonal to the fibers of $\pi$, that is, $\gamma$ should be a \emph{horizontal lift} of $c$:
\begin{align*}
    \gamma'(t) \perp \T_{\gamma(t)}(\pi^{-1}(c(t))) = \ker\D\pi(\gamma(t)).
\end{align*}
These two conditions together indeed fully determine the velocity of $\gamma$:
\begin{align}
    \gamma'(t) & = \D\pi(\gamma(t))^\dagger[c'(t)]
    \label{eq:gammaprimeidea}
\end{align}
where the dagger denotes the Moore--Penrose pseudoinverse.
Together with the initial condition $\gamma(0) = y$ (and some additional technical work), this yields a differential equation for $\gamma$.
We show that its solution exists for all $t \in [0, 1]$, so that $\gamma(1)$ is a well-defined function of $y$: we call it $\varphi(y)$.
Much of the proof serves the purpose of making sure that this is smooth in $y$ and has all the other required properties.
Notice $f$ is constant along $\gamma$ because
\begin{align*}
    (f \circ \gamma)'(t) = \D f(\gamma(t))[\gamma'(t)] = \inner{\grad f(\gamma(t))}{\gamma'(t)} = 0
\end{align*}
owing to the fact that $\gamma'$ is orthogonal to the fibers of $\pi$ while $\grad f$ is tangent to them.
This is how we get to conclude that $f(\varphi(y)) = f(y)$.


The proof of the following theorem formalizes these ideas.
It notably recovers Corollary~\ref{cor:fiberbundlegeneral} with added control over the value of $f$ through the trivializations, and readily extends to the globally trivial case under contractibility.

\begin{theorem} \label{thm:pitrivial}
    Let $f \colon \calM \to \reals$ be smooth and globally \pl{}.
    Its set $S$ of critical points is a smooth embedded submanifold of $\calM$, and the end-point map $\pi \colon \calM \to S$~\eqref{eq:pi} is a surjective smooth submersion.
    Fix a point $\bar x \in S$.
    Its fiber $\calF = \pi^{-1}(\bar x)$ is a smooth embedded submanifold of $\calM$.
    
    Let $U \subseteq S$ be a contractible (open) neighborhood of $\bar x$ (e.g., an appropriate chart domain).
    There exists a map $\varphi \colon \pi^{-1}(U) \to \calF$ such that
    \begin{itemize}
        \item The map $\psi \colon \pi^{-1}(U) \to U \times \calF \colon y \mapsto \psi(y) = (\pi(y), \varphi(y))$ is a diffeomorphism, and
        \item For all $y \in \pi^{-1}(U)$, we have $f(y) = f(\varphi(y))$.
    \end{itemize}
    Thus, $\pi$ is a smooth fiber bundle (Definition~\ref{def:smoothfiberbundle}).
    If $S$ is contractible (equivalently, if $\calM$ is so), the above holds with $U = S$ so that $\pi$ is a trivial smooth fiber bundle.
\end{theorem}
\begin{proof}
    The preliminary statements of the theorem follow from Propositions~\ref{prop:Ssmoothpisubmersion} and~\ref{prop:plrestrictedtofiber}.
    See Proposition~\ref{prop:picontinuous} for the claim that $S$ is contractible if and only if $\calM$ is contractible.

    \paragraph{The pseudoinverse of $\D\pi$:}
    Endow $S$ with a Riemannian structure, for example as a Riemannian submanifold of $\calM$. 
    For each point $y \in \calM$, consider the map $\D\pi(y) \colon \T_y \calM \to \T_{\pi(y)} S$.
    Let $\D\pi(y)^* \colon \T_{\pi(y)} S \to \T_y \calM$ denote its adjoint with respect to the Riemannian metrics on $\calM$ and $S$.
    Since $\pi$ is a submersion, $\D\pi(y)$ is surjective for all $y \in \calM$, hence we may define its Moore--Penrose pseudoinverse as
    \begin{align*}
        \D\pi(y)^\dagger = \D\pi(y)^* \circ (\D\pi(y) \circ \D\pi(y)^*)^{-1} \colon \T_{\pi(y)} S \to \T_y \calM.
    \end{align*}
    Notice that this depends smoothly on $y$.

    \paragraph{A smooth collection of paths in $U$:}
    Since $U$ is contractible (Definition~\ref{def:contractible}), it deformation retracts to $\bar x \in U$.
    Specifically, there exists a homotopy $H \colon U \times [0, 1] \to U$ between the identity map on $U$ and the constant map to $\{\bar x\}$:
    \begin{align*}
        H(x, 0) = x && \textrm{ and } && H(x, 1) = \bar x && \textrm{ for all } x \in U.
    \end{align*}
    We can choose $H$ to be smooth because $U$ is smooth as an open submanifold of $S$.
    This is because $H$ is a homotopy from the identity map $\mathrm{id} \colon U \to U$, $\mathrm{id}(x) = x$, to a constant map $c \colon U \to U$, $c(x) = \bar x$, and if two smooth maps are homotopic then they are \emph{smoothly} homotopic by Whitney's approximation theorem~\citep[Thm.~6.29]{lee2012smoothmanifolds}.
    Moreover, that theorem's proof (see reference) provides the existence of a smooth map $H \colon U \times \reals \to U$ with the properties stated above.
    Choose that $H$ going forward.

    \paragraph{Recalling the proof intuition:}
    Let $\calM' = \pi^{-1}(U)$: this is smooth as an open submanifold of $\calM$.
    We aim to build a diffeomorphism $y \mapsto \psi(y) = (\pi(y), \varphi(y))$ from $\calM'$ to $U \times \calF$, in such a way that $\varphi$ maps each fiber of $\pi$ (in $\calM'$) to the fiber $\calF$.
    To do so, given a point $y \in \calM'$, below, we build a curve $\gamma$ that brings $y = \gamma(0)$ to a point $\gamma(1) \in \calF$.
    This curve is a lift of a corresponding curve $c(t) = H(\pi(y), t)$ on $U$, which brings $c(0) = \pi(y)$ to $c(1) = \bar x$.
    By \emph{lift} we mean that $\pi \circ \gamma = c$.
    Moreover, we aim for a \emph{horizontal} lift in the sense that $\gamma'$ is orthogonal to the fibers of $\pi$.
    The plan is to let $\varphi(y) = \gamma(1)$.
    (See Figure~\ref{fig:PLflowbundle}.)

    \paragraph{A technical departure from that intuition:}
    While we would like to use~\eqref{eq:gammaprimeidea} to define a differential equation in $\gamma$, we cannot yet assume that such a $\gamma$ would indeed satisfy $\pi \circ \gamma  = c$, and so we cannot be certain that $c'(t)$ (a tangent vector at $c(t)$) would indeed be in the domain of $\D\pi(\gamma(t))^\dagger$ (that is, the tangent space at $\pi(\gamma(t))$).

    To make up for this, we invoke the existence of a smooth map
    \begin{align*}
        T \colon \T S \times S \to \T S \colon ((x_1, v), x_2) \mapsto T_{x_2 \leftarrow x_1}(v)
    \end{align*}
    with the following properties (see Appendix~\ref{app:transporter}; we call this a \emph{transporter}):
    \begin{enumerate}
        \item $v \mapsto T_{x_2 \leftarrow x_1}(v)$ is a linear map from $\T_{x_1} S$ to $\T_{x_2} S$ for all $x_1, x_2 \in S$, and
        \item $T_{x \leftarrow x}(v) = v$ for all $(x, v) \in \T S$.
    \end{enumerate}
    We use this smooth map to transport vectors from the tangent space at $c(t)$ to the tangent space at $\pi(\gamma(t))$, in such a way that if these two points turn out to be the same (they will), then the map has no effect.
    
    \paragraph{Setting up an ODE for $\gamma$:}
    With intuition driven by~\eqref{eq:gammaprimeidea} and now equipped with the transporter $T$, let $W \colon \calM' \times U \times \reals \to \T\calM'$ be defined as follows:
    \begin{align}
        W(y, x, t) = \D\pi(y)^\dagger\!\left[  T_{\pi(y) \leftarrow H(x, t)}\!\left[ \ddt H(x, t) \right]\right].
        \label{eq:Wforliftgamma}
    \end{align}
    The map $W$ is smooth because (a) $H$ is smooth, (b) $y \mapsto \D\pi(y)^\dagger$ is smooth, and (c) the transporter $T$ is smooth.
    
    This allows us to consider the following smooth, non-autonomous ODE in the unknown curves $\gamma$ on $\calM'$ and $\chi$ on $U$ (the constant curve $\chi$ is included for technical reasons):
    \begin{align}
        \ddt \gamma(t) = W\!\big(\gamma(t), \chi(t), t\big) && \textrm{ and } && \ddt \chi(t) = 0,
        \label{eq:ODEforliftgamma}
    \end{align}
    with the two following sets of initial conditions (to be considered separately):
    \begin{enumerate}
        \item Either $\gamma(0) = y \in \calM'$ and $\chi(0) = \pi(y)$,
        \item Or $\gamma(1) = z \in \calF$ and $\chi(1) = x \in U$ (with $z$ and $x$ to be specified later).
    \end{enumerate}
    We are mostly interested in what happens for $t$ in the interval $[0, 1]$.
    The first set of initial conditions corresponds to a curve $\gamma$ that starts at $y$ and ends at some point $\gamma(1)$, which we plan to identify with $\varphi(y)$.
    The second set is used later to construct the inverse of the map $\psi = (\pi, \varphi)$.
    In both cases, note that $\chi$ is constant (respectively equal to $\pi(y)$ or $x$ for all $t$).
    Also define the curve
    \begin{align}
        c(t) = H(\chi(t), t)
        \label{eq:cforliftgamma}
    \end{align}
    which starts at $c(0) = \chi(0)$ (that is, respectively $\pi(y)$ or $x$) and ends at $c(1) = \bar x$.

    For either set of initial conditions, the ODE admits a unique smooth solution $(\gamma, \chi)$ defined over a maximal interval of time that is open.
    Since $\chi$ is constant, we can focus on $\gamma$.
    Let us first argue that $\pi \circ \gamma = c$; then we show that $\gamma$ is defined (in particular) over the whole interval $[0, 1]$.

    \paragraph{The curve $\gamma$ is a horizontal lift of $c$:}
    We know $\gamma$ exists on some interval.
    Define $\eta = \pi \circ \gamma$ and compute
    \begin{align*}
        \eta'(t) = \D\pi(\gamma(t))\!\left[\gamma'(t)\right] & = \D\pi(\gamma(t))\!\left[ \D\pi(\gamma(t))^\dagger\!\left[T_{\pi(\gamma(t)) \leftarrow H(\chi(t), t)}\!\left[\ddt H(\chi(t), t)\right]\right]\right] \\
            & = T_{\eta(t) \leftarrow c(t)}\!\left[c'(t)\right],
    \end{align*}
    where
    the simplification occurred because $\D\pi(\gamma(t)) \circ \D\pi(\gamma(t))^\dagger$ is identity. 
    We can view this as an ODE in $\eta$ with the two following sets of initial conditions:
    \begin{enumerate}
        \item Either $\eta(0) = \pi(\gamma(0)) = \pi(y) = c(0)$,
        \item Or $\eta(1) = \pi(\gamma(1)) = \pi(z) = \bar x  = H(\chi(1), 1) = c(1)$.
    \end{enumerate}
    Either way, the solution exists and is unique.
    Of course, we already know $\pi \circ \gamma$ is a solution.
    Moreover, we see that $c$ is a solution as well, because $T_{\eta(t) \leftarrow c(t)}$ is identity if $\eta = c$.
    By uniqueness, we deduce that $\pi \circ \gamma = c$.

    Thus, we have found that $\gamma$ is a lift of $c$.
    Plugging $\pi \circ \gamma = c$ into the ODE~\eqref{eq:ODEforliftgamma} reveals that, for all $t$ in the domain of $\gamma$, we have
    \begin{align*}
        \gamma'(t) = \D\pi(\gamma(t))^\dagger\!\left[c'(t)\right].
    \end{align*}
    Notice also that $\gamma$ is a \emph{horizontal} lift of $c$ in the sense that
    \begin{align*}
        \gamma'(t) \in \im \D\pi(\gamma(t))^\dagger = \left( \ker \D\pi(\gamma(t)) \right)^\perp,
    \end{align*}
    that is, $\gamma'(t)$ is orthogonal to the tangent space of the fiber of $\pi$ passing through $\gamma(t)$.

    \paragraph{The curve $\gamma$ is defined over the whole interval $[0, 1]$:}
    This is the only part of this proof where we use the fact that $\pi$ is the end-point map of the negative gradient flow for $f$, rather than a general surjective smooth submersion.
    For starters, notice that
    \begin{align}
        \ddt f(\gamma(t)) = \innerbig{\grad f(\gamma(t))}{\gamma'(t)}_{\gamma(t)} = 0,
        \label{eq:ddtfgammazero}
    \end{align}
    because $\gamma'(t)$ is orthogonal to the fibers of $\pi$, while $\grad f(\gamma(t))$ is tangent to the fibers of $\pi$.
    Thus, $f(\gamma(t))$ is constant for all $t$ in the domain of $\gamma$: let $\bar f$ denote that constant (equal to $f(\gamma(0)) = f(y)$ or $f(\gamma(1)) = f(z)$, depending on the set of initial conditions).
    Let $\mu > 0$ be the \pl{} constant of $f$ and let $f^* = \inf f$.
    By the quadratic growth property (Lemma~\ref{lemma:bound-trajectories-QG}), with $\dist$ denoting the Riemannian distance on $\calM$ (to be clear, not $\calM'$) we have 
    \begin{align*}
        \bar f - f^* = f(\gamma(t)) - f^* \geq \frac{\mu}{2} \dist\!\big(\gamma(t), \pi(\gamma(t))\big)^2 = \frac{\mu}{2} \dist\!\big(\gamma(t), c(t)\big)^2.
    \end{align*}
    Let $\ell$ denote the length of the curve $c$ over the interval $[0, 1]$ (in the metric of $\calM$).
    Then, $\dist(c(t), c(1)) \leq \ell$ holds for all $t \in [0, 1]$, and it follows that
    \begin{align*}
        \dist\!\big(\gamma(t), \bar x\big) \leq \dist\!\big(\gamma(t), c(t)\big) + \dist\!\big(c(t), c(1)\big) \leq \sqrt{\frac{2 (\bar f - f^*)}{\mu}} + \ell
    \end{align*}
    for all $0 \leq t \leq 1$ in the domain of $\gamma$.
    In other words, $\gamma|_{[0, 1]}$ remains in a closed ball $B$ of finite radius around $\bar x$ (in the metric of $\calM$).
    This is a compact set since $\calM$ is complete.
    We also know from $\pi(\gamma(t)) = c(t)$ that $\gamma|_{[0, 1]}$ stays in $C := \pi^{-1}(c([0, 1]))$.
    This is a closed set of $\calM$ (because $c([0, 1])$ is compact as the continuous image of a compact set, and $\pi$ is continuous so the pre-image of a closed set is closed).
    Also, $C$ is entirely contained in $\calM'$.
    Therefore, $\gamma|_{[0, 1]}$ remains in $B \cap C$, which is a compact set of $\calM$ contained in $\calM'$
    and hence it is compact in $\calM'$. 
    Therefore, the escape lemma~\citep[Lem.~9.19]{lee2012smoothmanifolds} guarantees $\gamma$ is defined over the whole interval $[0, 1]$.\footnote{The escape lemma in that reference is stated for autonomous ODEs. The result extends to non-autonomous ODEs by the standard trick which consists in adding a curve $\tau$ on $\reals$ to the system, with $\tau'(t) = 1$ and $\tau(0) = 0$ or $\tau(1) = 1$. Then, any occurrence of $t$ can be replaced by $\tau(t)$.}

    \paragraph{Defining $\varphi$ and $\psi$:}
    Now, using the first set of initial conditions, we can define $\varphi(y) = \gamma(1)$ as intended.
    Of course, $\gamma(1)$ is in $\calF$ because $\pi(\gamma(1)) = c(1) = \bar x$.
    Also,
    \begin{align*}
        f(\varphi(y)) = f(\gamma(1)) = f(\gamma(0)) = f(y)
    \end{align*}
    owing to~\eqref{eq:ddtfgammazero}.
    By the fundamental theorem of time-dependent flows~\citep[Thm.~9.48]{lee2012smoothmanifolds} applied to~\eqref{eq:ODEforliftgamma}, $\varphi$ is smooth.
    Thus, $\psi = (\pi, \varphi) \colon \calM' \to U \times \calF$ is smooth.

    \paragraph{Showing $\psi$ is a diffeomorphism:}
    Let us build the inverse of $\psi$ and argue it is smooth.
    Intuitively, the idea is to run the ODE~\eqref{eq:ODEforliftgamma} in reverse.

    Precisely, for a given $(x, z)$ in $U \times \calF$, solve~\eqref{eq:ODEforliftgamma} with the \emph{second} set of initial conditions: these fix the curves at $t = 1$ rather than $t = 0$.
    The solution provides a constant curve $\chi(t) = x$ and a curve $\gamma \colon [0, 1] \to \calM'$ such that $\gamma(1) = z$ and
    \begin{align*}
        \pi(\gamma(0)) = c(0) = H(\chi(0), 0) = \chi(0) = \chi(1) = x.
    \end{align*}
    Thus, $\gamma(0)$ belongs to the fiber of $x$.
    Let $\xi \colon U \times \calF \to \calM'$ be defined as $\xi(x, z) = \gamma(0)$.
    This map too is smooth, for the same reason that $\varphi$ is smooth.

    Let us check that $\xi$ is the inverse of $\psi = (\pi, \varphi)$.
    For all $(x, z) \in U \times \calF$, we have
    \begin{align*}
        \pi(\xi(x, z)) = \pi(\gamma(0)) = x.
    \end{align*}
    To see that also $\varphi(\xi(x, z)) = z$, reason as follows.
    Let $\gamma_a \colon [0, 1] \to \calM'$, $\chi_a \colon [0, 1] \to U$ be the solution of~\eqref{eq:ODEforliftgamma} with initial conditions $\gamma_a(1) = z$ and $\chi_a(1) = x$.
    These are such that $\xi(x, z) = \gamma_a(0)$.
    Now let $y = \xi(x, z)$, and
    let $\gamma_b \colon [0, 1] \to \calM'$, $\chi_b \colon [0, 1] \to U$ be the solution of~\eqref{eq:ODEforliftgamma} with initial conditions $\gamma_b(0) = y$ and $\chi_b(0) = \pi(y)$.
    These are such that $\varphi(y) = \gamma_b(1)$.
    Notice that
    \begin{align*}
        \gamma_a(0) = \xi(x, z) = y = \gamma_b(0) && \textrm{ and } && \chi_a(0) = \chi_a(1) = x = \pi(\xi(x, z)) = \pi(y) = \chi_b(0).
    \end{align*}
    Thus, $\gamma_a$ and $\chi_a$ are the same as $\gamma_b$ and $\chi_b$, by uniqueness of solutions for ODEs.
    Consequently,
    \begin{align*}
        \varphi(\xi(x, z)) = \varphi(y) = \gamma_b(1) = \gamma_a(1) = z.
    \end{align*}
    Overall, we have shown that $\psi(\xi(x, z)) = (\pi(\xi(x, z)), \varphi(\xi(x, z))) = (x, z)$ for all $(x, z)$ in $U \times \calF$.
    For the same reason, $\xi(\psi(y)) = y$ for all $y \in \calM'$.
    This concludes the proof that $\psi$ is a diffeomorphism from $\calM'$ to $U \times \calF$, with $\xi$ as its smooth inverse.
\end{proof}


\subsection{Combining the pieces} \label{sec:proofplnonlinlstsq}

We are now ready to prove Theorem~\ref{thm:plnonlinlstsq}.
It is a corollary of the following more general statement, because under the contractibility assumption we can let $U = S$ and note that $\pi^{-1}(U) = \calM$.

\begin{theorem} \label{thm:plnonlinlstsqnocontractibility}
  Let $f \colon \calM \to \reals$ be smooth and globally \pl{}.
  Its set $S$ of critical points is a connected, properly embedded smooth submanifold of $\calM$.

  Let $U$ be a contractible, open subset of $S$.
  There exists a diffeomorphism $\psi \colon \pi^{-1}(U) \to U \times \Rk$ of the form $\psi = (\pi, \varphi)$ such that $f(y) = f^* + \|\varphi(y)\|^2$ for all $y \in \pi^{-1}(U)$, where $f^* = \inf_{y \in \calM} f(y)$.
\end{theorem}
\begin{proof}
  The properties of $S$~\eqref{eq:S} follow from Propositions~\ref{prop:picontinuous} and~\ref{prop:Ssmoothpisubmersion}.

  Let $\calM' = \pi^{-1}(U)$.
  Fix $\bar x \in U$ to invoke Theorem~\ref{thm:pitrivial}.
  This yields a diffeomorphism $\tilde \psi = (\pi, \varphi_1) \colon \calM' \to U \times \calF$ with $\calF = \pi^{-1}(\bar x)$ such that $f(y) = f(\varphi_1(y))$ for all $y \in \calM'$.

  The restriction of $f$ to $\calF$ is \pl{} with $\bar x$ as its unique critical point: see Proposition~\ref{prop:plrestrictedtofiber}.
  Notice that $k = \dim \calM' - \dim U = \dim \calM - \dim S = \dim \calF$.
  Thus, applying Theorem~\ref{thm:globalsquaredistsinglecppl} to $f|_{\calF}$ provides a diffeomorphism $\varphi_2 \colon \calF \to \Rk$ such that $f(y) = f^* + \|\varphi_2(y)\|^2$ for all $y \in \calF$.

  Compose these diffeomorphisms to form $\psi = (\pi, \varphi_2 \circ \varphi_1) \colon \calM' \to U \times \Rk$.
  This is indeed an appropriate diffeomorphism
  because
  \begin{align*}
    f(y) = f(\varphi_1(y)) = f^* + \|\varphi_2(\varphi_1(y))\|^2
  \end{align*}
  for all $y \in \calM'$, as required.
\end{proof}

Corollary~\ref{cor:deformintoquadform} is now a consequence of the more general result below, because if $S$ is diffeomorphic to $\reals^m$, we may take $U = S$.

\begin{corollary}\label{cor:deformintoquadformMoregeneral}
Let $f \colon \calM \to \reals$ be smooth and globally \pl{}, and
let $U$ be an open subset of $S$ which is diffeomorphic to $\reals^m$.
There exists a diffeomorphism $\xi \colon \pi^{-1}(U)  \to \reals^n$ such that
\[
f(\xi^{-1}(y)) \;=\; f^* + y_{m+1}^2 + \cdots + y_n^2, \quad \quad \forall y \in \reals^n,
\]
\end{corollary}
\begin{proof}
This follows from Theorem \ref{thm:plnonlinlstsqnocontractibility} by chaining diffeomorphisms.
Let $\sigma \colon U \to \reals^m$ be a diffeomorphism, and define the diffeomorphism
$$\Sigma \colon U \times \reals^k \to \reals^m \times \reals^k = \reals^n, \quad \quad \Sigma(w, z) = (\sigma(w), z).$$
Theorem \ref{thm:plnonlinlstsqnocontractibility} provides a diffeomorphism $\psi = (\pi, \varphi) \colon \pi^{-1}(U) \to U \times \Rk$
such that 
$$\text{$f(\psi^{-1}(w, z)) = f^* + \|z\|^2$ for all $(w, z) \in U \times \Rk$.}$$
Therefore, $\xi = \Sigma \circ \psi$ is a diffeomorphism from $\pi^{-1}(U)$ to $\reals^m \times \reals^k$ satisfying 
$$f(\xi^{-1}(v, z)) = f(\psi^{-1}(\sigma^{-1}(v), z)) = f^* + \|z\|^2, \quad \quad \forall (v, z) \in \reals^m \times \reals^k,$$
as desired.
\end{proof}


\section{Building \pl{} functions} \label{sec:corollaries}

To prove Theorem~\ref{thm:constructingglobalPLcor}, we must explicitly \emph{construct} a globally \pl{} function $f \colon \calM \to \reals$ whose set of minimizers matches a given submanifold $S$.
A key subtlety is that $f$ must be globally \pl{} with respect to the \emph{given} Riemannian metric on $\calM$---the metric cannot be altered (which would make the problem substantially easier).  
We propose such a construction below.

\begin{proof}[Proof of Theorem~\ref{thm:constructingglobalPLcor}]
The given diffeomorphism $\psi \colon \calM \to S \times \Rk$ has two parts: $\psi = (\psi_1, \psi_2)$ with $\psi_1 \colon \calM \to S$ and $\psi_2 \colon \calM \to \Rk$.
Introduce the smooth map $c \colon \reals \times \calM \to \calM$ defined by $c(t, y) = \psi^{-1}(\psi_1(y), t\psi_2(y))$.
For convenience, let $c_y(t) = c(t, y)$: this is a smooth curve on $\calM$ which travels from $c_y(0)$ (a point on $S$) to $c_y(1) = y$.

Define $f(y)$ to be the integral of the squared speed of that curve (in the metric of $\calM$):
\begin{align*}
  f(y) = \int_0^1 \|c_y'(t)\|^2 \dt.
\end{align*}
This function $f \colon \calM \to \reals$ is smooth because $c$ is smooth.
Moreover, $f$ is nonnegative, and $f(y) = 0$ if and only if $y$ is in $S$.
Indeed, $c_y'(t) = \D\psi^{-1}(\psi_1(y), t \psi_2(y))[0, \psi_2(y)]$ and $\psi$ is a diffeomorphism so $\D\psi^{-1}$ is invertible at every point; it follows that $c_y'(t) = 0$ if and only if $\psi_2(y) = 0$, which holds if and only if $y \in S$.
Thus, the set of minimizers of $f$ is exactly $S$.
It remains to show that $f$ is globally \pl{}.
To this end, fix an arbitrary $y \in \calM \backslash S$.

Since $c_y(1) = y$, we have that $c_y'(1)$ is a nonzero tangent vector to $\calM$ at $y$.
Then, we can compute the directional derivative of $f$ at $y$ along $c_y'(1)$ as:
\begin{align*}
  \D f(y)[c_y'(1)] = (f \circ c_y)'(1) = \left. \dds  \int_0^1 \|c_{c_y(s)}'(t)\|^2 \dt \, \right|_{s = 1}.
\end{align*}
The key observation here is this:
\begin{align*}
  c_{c_y(s)}(t) = c(t, c_y(s)) = \psi^{-1}(\psi_1(c_y(s)), t \psi_2(c_y(s))) = \psi^{-1}(\psi_1(y), t s \psi_2(y)) = c_{y}(t s).
\end{align*}
Therefore, we also have
\begin{align*}
  c_{c_y(s)}'(t) = \ddt c_y(t s) = s \cdot c_y'(t s).
\end{align*}
This allows us to continue the computation of the directional derivative: we first substitute the above expression, and then change the integration variable $t$ in favor of $\tau = t s$ (so that $\dtau = s \dt$ and the integration limits become $0$ to $s$):
\begin{align*}
  \D f(y)[c_y'(1)] & = \left. \dds  \int_0^1 s^2 \|c_y'(t s)\|^2 \dt \, \right|_{s = 1} \\
  & = \left. \dds  s \int_0^s \|c_y'(\tau)\|^2 \dtau \, \right|_{s = 1} \\
  & = \int_0^1 \|c_y'(\tau)\|^2 \dtau + \|c_y'(1)\|^2 = f(y) + \|c_y'(1)\|^2.
\end{align*}
To conclude, we use the Cauchy--Schwarz inequality to write
\begin{align*}
  \|\nabla f(y)\| \|c_y'(1)\| \geq \D f(y)[c_y'(1)] = f(y) + \|c_y'(1)\|^2.
\end{align*}
Now divide by $\|c_y'(1)\|$, square, and use the inequality $(a + b)^2 \geq 2ab$ to deduce
\begin{align*}
  \|\nabla f(y)\|^2 \geq \left( \frac{f(y)}{\|c_y'(1)\|} + \|c_y'(1)\| \right)^2 \geq 2 f(y).
\end{align*}
This confirms that $f$ is globally $1$-\pl{}, which concludes the proof.
\end{proof}

\section{Changing the metric to gain geodesic convexity} \label{sec:gconvexproof}

We here prove Theorem~\ref{thm:pl-gconvex} from Section~\ref{sec:gconvexintro}, which states (essentially) that if $f \colon \calM \to \reals$ is a smooth, globally \pl{} function and $\calM$ is contractible then $\calM$ can be given a new, complete Riemannian metric such that $f$ is still globally \pl{} but it is now also geodesically convex.

\begin{proof}[Proof of Theorem~\ref{thm:pl-gconvex}]
The qualities of $S$ are as provided by Theorem~\ref{thm:plnonlinlstsq}.

Let us start with part~\ref{item:pl-gconvex-basic}.
Endow $S$ with the Riemannian metric it inherits from $\calM$: this is a complete metric because $S$ is a closed subset of $\calM$.
Equip $\reals^k$ with the standard Euclidean metric, and give $S \times \reals^k$ the product metric: it is complete.

If $\calM$ is contractible, Theorem~\ref{thm:plnonlinlstsq} provides a diffeomorphism $\psi = (\pi, \varphi) \colon \calM \to S \times \reals^k$ such that (after cosmetic rescaling)
\begin{equation*}
  f(\psi^{-1}(w,z)) \;=\; f^* + \tfrac{1}{2} \|z\|^2,
  \qquad \forall (w, z) \in S \times \reals^k.
\end{equation*}
We claim that $f \circ \psi^{-1}$ is g-convex on $S \times \reals^k$ under the product metric.

Indeed, let $\gamma = (\gamma_1,\gamma_2) \colon [0, 1] \to S \times \reals^k$ be any geodesic segment.
Then $\gamma_1$ and $\gamma_2$ are geodesics in $S$ and $\reals^k$, respectively~\citep[Pb.~5-7]{lee2018riemannian}.
In particular, $\gamma_2$ is affine.
Since $z \mapsto \|z\|^2$ is convex on $\reals^k$, the map $t \mapsto \|\gamma_2(t)\|^2$ is convex on $[0, 1]$.
Thus, $f \circ \psi^{-1} \circ \gamma$ is convex, because
\begin{equation*}
  f(\psi^{-1}(\gamma(t))) = f(\psi^{-1}(\gamma_1(t), \gamma_2(t))) = f^* + \tfrac{1}{2} \|\gamma_2(t)\|^2.
\end{equation*}
Therefore, $f \circ \psi^{-1}$ is g-convex on $S \times \reals^k$.
This function is also globally 1-\pl{} on $S \times \reals^k$ since $\grad(f \circ \psi^{-1})(w, z) = (0, z)$.

Finally, pull back the product metric via $\psi$ to obtain a metric $\inner{\cdot}{\cdot}_2$ on $\calM$.
By design, $f$ is g-convex and 1-\pl{} with respect to $\inner{\cdot}{\cdot}_2$.

For the ``if'' direction of part~\ref{item:pl-gconvex-S-Rm}, reason as above, but call upon Corollary~\ref{cor:deformintoquadform} to provide the diffeomorphism $\xi \colon \calM \to \reals^n$ such that $f \circ \xi^{-1}$ is a convex quadratic, and pull back the Euclidean metric from $\reals^n$ to $\calM$ via $\xi$ to obtain $\inner{\cdot}{\cdot}_2$.
For the ``only if'' direction, observe that if $\calM$ (with its new metric) is isometric to $\Rn$ then there exists a diffeomorphism $\xi \colon \calM \to \Rn$ such that $\inner{\cdot}{\cdot}_2$ is the pullback of the Euclidean metric via $\xi$, and the assumption is that $f \circ \xi^{-1}$ is convex and globally \pl{}.
Its set of minimizers $C \triangleq \xi(S)$ is a smooth embedded submanifold of $\Rn$ that is also a closed and convex set. 
Thus, $C$ is an affine subspace of $\Rn$.\footnote{To see this, fix $x \in C$ and observe for all $y \in C$ that $c(t) = x + t(y-x)$ is a smooth curve on $C$ for $t \in [0, 1]$ (by convexity), hence $c'(0) = y-x$ is in $\T_x C$, that is, $C$ is included in the affine space $A \triangleq x + \T_x C$; moreover, $C$ is closed in $A$ (in subspace topology), and $C$ is open in $A$ (because it is an embedded submanifold of $A$ with $\dim C = \dim A$); therefore, $C = A$.}
It follows that $S$ is diffeomorphic to $\reals^m$.
\end{proof}

\section{A comment about contractibility} \label{sec:commentcontractible}

As usual, let $S$ be the set of minimizers of a smooth function $f \colon \calM \to \reals$ that is globally \pl{}.
If $\calM$ is contractible, then Theorem~\ref{thm:plnonlinlstsq} notably provides that $\calM$ is diffeomorphic to $S \times \Rk$.

One may ask: without assuming that $\calM$ is contractible, may it still be the case that the existence of such a function $f$ implies that $\calM$ is diffeomorphic to $S \times \Rk$?
We discuss here, with a summary in Table~\ref{tab:noncontractibleexamples}.

\begin{table}[t]
\begin{tabular}{@{}lcccccc@{}}
\toprule
                           &         &           & contractible & parallelizable & orientable  & $\calM \cong S \times \Rk$ \\ \midrule
\multirow{2}{*}{Example~\ref{ex:cylinder}} & $\calM$ & cylinder  & \multirow{2}{*}{\usym{2717}}  & \usym{2713}    & \usym{2713} & \multirow{2}{*}{\usym{2713}}                                               \\
                           & $S$     & circle    &   & \usym{2713}    & \usym{2713} &                                                                            \\ \midrule
\multirow{2}{*}{Example~\ref{ex:mobius}} & $\calM$ & M\"obius  & \multirow{2}{*}{\usym{2717}}  & \usym{2717}    & \usym{2717} & \multirow{2}{*}{\usym{2717}}                                               \\
                           & $S$     & circle    &   & \usym{2713}    & \usym{2713} &                                                                            \\ \midrule
\multirow{2}{*}{Example~\ref{ex:TS2}} & $\calM$  & $\T\Stwo$ & \multirow{2}{*}{\usym{2717}}  & \usym{2713}    & \usym{2713} & \multirow{2}{*}{\usym{2717}}                                               \\
                           & $S$     & 2-sphere  &  & \usym{2717}    & \usym{2713} &                                                                            \\ \midrule
\multirow{2}{*}{Example~\ref{ex:SoneTStwo}} & $\calM$  & $\Sone \times \T\Stwo$ & \multirow{2}{*}{\usym{2717}}  & \usym{2713}    & \usym{2713} & \multirow{2}{*}{\usym{2717}}                                               \\
                           & $S$     & $\Sone \times \Stwo$  &  & \usym{2713}    & \usym{2713} &                                                                            \\ \bottomrule
\end{tabular}
\caption{If $\calM$ is contractible, Theorem~\ref{thm:plnonlinlstsq} provides that $\calM$ is diffeomorphic ($\cong$) to $S \times \Rk$, with $S$ the set of minimizers of a smooth, globally \pl{} function. Can the assumption be relaxed? In Section~\ref{sec:commentcontractible}, we provide four examples of a globally \pl{} function $f \colon \calM \to \reals$ with a non-contractible domain, and check whether that conclusion of Theorem~\ref{thm:plnonlinlstsq} holds nonetheless. The first two have the same set of minimizers $S$ (a circle, up to diffeomorphism), yet the outcomes differ. Thus, assumptions on $S$ alone may not distinguish between the two. Assumptions on $\calM$ might, but the last two examples show it is not enough for $\calM$ and $S$ to be parallelizable.}
\label{tab:noncontractibleexamples}
\end{table}

In Example~\ref{ex:cylinder} below, we construct a (not constant) globally \pl{} function on a cylinder, with $S$ diffeomorphic to the circle $\Sone$.
And indeed, the cylinder is not contractible, yet it is diffeomorphic to $\Sone \times \reals$.
More generally, let $f$ be any smooth and globally \pl{} function on the cylinder.
By Theorem~\ref{thm:plnonlinlstsq}, its set of minimizers $S$ must be a smooth, properly embedded, connected submanifold of the cylinder.
A priori, it can have dimension 0, 1 or 2.
Dimension 2 forces $S$ to be the whole cylinder, in which case we do have a diffeomorphism for trivial reasons.
Dimension 0 is excluded because $S$ would then have to be a point; in particular, $S$ would be contractible, which would imply that the cylinder is contractible, but it is not.
This leaves dimension 1, that is, $S$ must be a smooth curve embedded on the cylinder.
From the classification of 1-manifolds~\citep[Pb.~15-13]{lee2012smoothmanifolds}, it follows that $S$ is diffeomorphic to $\Sone$ or to $\reals$. 
The latter is contractible, hence excluded for the same reason as the point.
It follows that $S$ is diffeomorphic to $\Sone$ and, as stated earlier, the cylinder is diffeomorphic to $\Sone \times \reals$.

In light of this first example, we refine the question as follows: can the contractibility assumption on $\calM$ be relaxed in a way that the cylinder case described above would be included as well?

This possibility is limited by Example~\ref{ex:mobius}.
There, we construct a smooth, globally \pl{} function on the M\"obius band in such a way that the set of minimizers is also diffeomorphic to $\Sone$.
Yet, famously, the M\"obius band is \emph{not} diffeomorphic to $\Sone \times \reals$.

Considering both of those examples, we find that their solution sets $S$ are diffeomorphic (both are circles), yet they yield different conclusions as to the existence of a diffeomorphism from $\calM$ to $S \times \Rk$.
It follows that if we were to replace the contractibility assumption on $\calM$ by any other assumption on $S$ (at least, one that is invariant under diffeomorphism) then we would be unable to distinguish between the first two examples.

Since $\calM$ and $S$ are homotopy equivalent (Proposition~\ref{prop:picontinuous}), this further implies that any assumption on $\calM$ that is a homotopy invariant would be unable to correctly allow for the cylinder while also correctly excluding the M\"obius band.

Thus, we should entertain relaxations of contractibility that are not homotopy invariants.
Further scrutiny of the two examples above suggest that we consider whether $\calM$ is \emph{parallelizable} or \emph{orientable}.
The following implications are classical:
\begin{align*}
  \textrm{contractible} && \implies && \textrm{parallelizable} && \implies && \textrm{orientable}.
\end{align*}
(The first implication holds because ``parallelizable'' means the tangent bundle is trivial, and as noted earlier any vector bundle over a contractible base space is trivial; the second implication is stated in~\citep[Prop.~15.17]{lee2012smoothmanifolds} together with definitions of both concepts.)

This direction too is unfruitful.
Example~\ref{ex:TS2} defines a globally \pl{} function on the tangent bundle of the 2-sphere $\Stwo$ (that is, $\calM = \T\Stwo$) with set of minimizers $S$ diffeomorphic to $\Stwo$.
In this case, $\calM$ is not contractible, but it is parallelizable~\citep[Thm.~2.5, Thm.~3.2]{fodor2019parallelizability}.\footnote{See also \url{https://mathoverflow.net/questions/500443}.}
In contrast, $\Stwo$ itself is not parallelizable, and indeed its tangent bundle $\calM$ is not diffeomorphic to $\Stwo \times \reals^2$.
Thus, while $\calM$ is parallelizable, it is not diffeomorphic to $S \times \reals^2$.

Looking at the first three rows in Table~\ref{tab:noncontractibleexamples}, one might then hypothesize that perhaps it is enough for \emph{both} $\calM$ and $S$ to be parallelizable.
However, this too is insufficient as per Example~\ref{ex:SoneTStwo}.

\begin{example}[Cylinder over circle] \label{ex:cylinder}
  Let $\calM = \{x \in \reals^3 : x_1^2 + x_2^2 = 1\}$ be the cylinder as a Riemannian submanifold of $\reals^3$.
  The function $f(x) = x_3^2$ is globally 2-\pl{} since $\grad f(x) = (0, 0, 2x_3)$ and $\|\grad f(x)\|^2 = 4x_3^2 = 4f(x)$.
  The solution set $S = \{(x_1, x_2, 0) \in \calM\}$ is a circle, which is not contractible.
  Yet, the diffeomorphism $\psi \colon \calM \to S \times \reals \colon x \mapsto ((x_1, x_2, 0), x_3)$ is compatible with (what would be) the conclusions of Theorem~\ref{thm:plnonlinlstsq}.
\end{example}

\begin{example}[M\"obius over circle] \label{ex:mobius}
  Let $\calM = \reals^2 / \mathbb{Z}$ be the \emph{open M\"obius band}, that is, the quotient space where $\mathbb{Z}$ acts on $\reals^2$ by $n \cdot x = (x_1 + n, (-1)^n x_2)$.
  Give $\calM$ the smooth Riemannian manifold structure such that the quotient map $q \colon \reals^2 \to \calM$ is a normal Riemannian covering~\citep[Prop.~2.32, Ex.~2.35]{lee2018riemannian}. 
  In particular, $q$ is a local diffeomorphism~\citep[Prop.~4.33]{lee2012smoothmanifolds}
  and the Euclidean metric on $\reals^2$ is the pullback of the metric on $\calM$ through $q$, that is, for all $u, v \in \reals^2$ (thought of as tangent vectors to $\reals^2$ at $x$), we have
  \begin{align*}
    u^\top v = \inner{u}{v}_x^{\reals^2} = \inner{\D q(x)[u]}{\D q(x)[v]}_{q(x)}^{\calM}.
  \end{align*}
  Note that $\calM$ is non-empty,
  connected 
  and complete, 
  but it is famously not orientable~\citep[Ex.~10.3, Ex.~15.38]{lee2012smoothmanifolds}.

  Let $g \colon \reals^2 \to \reals \colon x \mapsto g(x) = x_2^2$.
  This function is invariant on the orbits of $\mathbb{Z}$, hence it descends to a well-defined smooth function $f \colon \calM \to \reals$ such that $g = f \circ q$~\citep[Thm.~4.29]{lee2012smoothmanifolds}.

  The minimal value of $f$ is zero, and the set of minimizers is $S = \{q(x_1, 0) \in \calM \colon x_1 \in \reals \}$.
  This is diffeomorphic to the circle $\Sone$~\citep[Ex.~10.3]{lee2012smoothmanifolds}. 

  One can check that the gradient of $f$ satisfies $\grad f(q(x)) = \D q(x)[\grad g(x)]$.
  For example, proceed by identification in the identity below which holds for all $u \in \reals^2$:
  \begin{multline*}
    \inner{\D q(x)[u]}{\D q(x)[\grad g(x)]}_{q(x)}^\calM = \inner{u}{\grad g(x)}_x^{\reals^2} = \D g(x)[u] \\ = \D f(q(x))[\D q(x)[u]] = \inner{\D q(x)[u]}{\grad f(q(x))}_{q(x)}^\calM.
  \end{multline*}
  In particular, from $\grad g(x) = (0, 2x_2)$ it follows that
  \begin{align*}
    \|\grad f(q(x))\|_{q(x)}^2 = \|\grad g(x)\|^2 = 4x_2^2 = 4 f(q(x))
  \end{align*}
  and hence $f$ is globally 2-\pl{} on $\calM$.

  Yet, the conclusions of Theorem~\ref{thm:plnonlinlstsq} could not possibly hold.
  Indeed, if they did, then there would exist a diffeomorphism from $\calM$ (the M\"obius band) to the product space $\Sone \times \reals$ (a cylinder).
  Yet, the latter is orientable while the former is not.
\end{example}

\begin{example}[Tangent bundle over sphere] \label{ex:TS2}
  Let $\calM = \T\Stwo = \{ (x, v) \in \reals^3 \times \reals^3 : x^\top x = 1 \textrm{ and } x^\top v = 0 \}$.
  This 4-dimensional manifold is the tangent bundle of the sphere $\Stwo$ in $\reals^3$: it is orientable and parallelizable~\citep[Thm.~2.5, Thm.~3.2]{fodor2019parallelizability} but not contractible. 
  Endow $\calM$ with the Riemannian submanifold metric $\inner{(\dot x, \dot v)}{(\ddot x, \ddot v)}_{(x, v)} = \dot x^\top \ddot x + \dot v^\top \ddot v$.
  (The example also works with the Sasaki metric, see below.) 
  Notice that $\calM$ is indeed connected and complete~\citep[13-18(b)]{lee2012smoothmanifolds}. 

  Let $f \colon \calM \to \reals$ be defined by $f(x, v) = \frac{1}{2} v^\top v$.
  This is clearly smooth ($C^\infty$).
  Its set of minimizers $S = \{ (x, 0) : x^\top x = 1 \}$ is diffeomorphic to the sphere $\Stwo$ (also orientable but neither parallelizable nor contractible).
  Moreover, the gradient of $f$ on $\calM$ is $\nabla f(x, v) = (0, v)$ because $\D f(x, v)[\dot x, \dot v] = v^\top \dot v = \inner{(\dot x, \dot v)}{(0, v)}_{(x, v)}$ for all $(\dot x, \dot v)$ in the tangent space to $\calM$ at $(x, v)$.
  It follows that
  \begin{align*}
    \|\nabla f(x, v)\|_{(x, v)}^2 = \inner{(0, v)}{(0, v)}_{(x, v)} = v^\top v = 2f(x, v),
  \end{align*}
  hence $f$ is globally 1-\pl{}.

  If the contractibility assumption could be removed in Theorem~\ref{thm:plnonlinlstsq}, then we would obtain here a diffeomorphism from $\calM$ to $S \times \reals^2$.
  Yet, that is impossible because $\calM$ (the tangent bundle of $\Stwo$) is not even homeomorphic to $\Stwo \times \reals^2$.\footnote{See for example \href{https://mathoverflow.net/a/209205/100537}{mathoverflow.net/a/209205/100537}.}
\end{example}

\begin{example} \label{ex:SoneTStwo}
  Modifying the previous example, let $\calM = \Sone \times \T\Stwo$ (with the Riemannian submanifold metric).
  The smooth, globally \pl{} function $f(x, (y, v)) = \frac{1}{2}v^\top v$ on $\calM$ has a set of minimizers $S = \{ (x, (y, 0)) \in \calM \}$ which is diffeomorphic to $\Sone \times \Stwo$.
  Notice that $\calM$ is parallelizable (as it is a product of two parallelizable manifolds), and likewise, $S$ is parallelizable (because it is a product of spheres, one of which has odd dimension~\citep[Thm.~XII]{kervaire1956thesis}).
  However, $\calM$ is not homeomorphic to $S \times \reals^2$.
  One way to verify this is to define a topological property, then to check that it is invariant under homeomophism, and show that $\calM$ has that property whereas $S \times \reals^2$ does not.
  Explicitly, the property to consider for a topological space $Z$ is as follows: For all compact $K \subseteq Z$, there exists a compact $K' \subseteq Z$ such that (a) $K \subseteq K'$,  (b) $Z \backslash K'$ is path-connected, and (c) the fundamental group of $Z \backslash K'$ does not contain a subgroup isomorphic to $\mathbb{Z} \times \mathbb{Z}$.
\end{example}

Examples~\ref{ex:cylinder} and~\ref{ex:TS2} generalize as follows.
Let $\calN$ be any (complete and connected) Riemannian manifold.
Let $\calM = \T\calN = \{(x, v) : x \in \calN \textrm{ and } v \in \T_x\calN \}$ be the tangent bundle of $\calN$, endowed with the Sasaki metric so that it is itself a (complete and connected) Riemannian manifold. 
Consider the smooth function $f \colon \calM \to \reals$ defined by
$$
  f(x, v) = \frac{1}{2} \|v\|_x^2.
$$
Its minimal value is zero, attained exactly on the so-called zero section $\{ (x, 0) : x \in \calN \}$, which is diffeomorphic to $\calN$.
Every tangent vector to $\T\calN$ at $(x, v)$ can be realized as the initial velocity of a smooth curve $c(t) = (x(t), v(t))$ on $\T\calN$ with $c(0) = (x, v)$.
Then,
$$
  \D f(x, v)[c'(0)] = (f \circ c)'(0) = \frac{1}{2} \left.\ddt \|v(t)\|_{x(t)}^2 \right|_{t=0} = \innerBig{v(0)}{\Ddt v(0)}_{x(0)} = \inner{(0, v)}{c'(0)}_{(x, v)},
$$
where $\Ddt$ denotes the covariant derivative on $\calN$, and in the last step we used the definition of the Sasaki metric~\citep[eq.~(1.1)]{musso1988tangentbundle}.
Thus, $\grad f(x, v) = (0, v)$ and $\|\grad f(x, v)\|_{(x, v)}^2 = \|v\|_x^2 = 2f(x, v)$ so that $f$ is globally 1-\pl{} with $S$ diffeomorphic to $\calN$.


%
%
%
%

\section{Perspectives} \label{sec:conclusions}

We conclude with a list of open questions.

\begin{itemize}

\item \textbf{Beyond the global \pl{} assumption on $f$.}
The global \pl{} condition is a convenient structural hypothesis, yet some of
our arguments rely only on weaker ingredients.
For instance, Theorem~\ref{thm:globalsquaredistsinglecppl} ultimately uses
coercivity together with the existence of a unique, nondegenerate critical point.
More generally, can Theorem~\ref{thm:plnonlinlstsq} be extended to functions
whose critical set is a Morse--Bott manifold of global minimizers, possibly
under uniform curvature or coercivity assumptions along normal directions?
What are the minimal assumptions that still yield comparable global geometric
conclusions?

\item \textbf{Beyond the contractibility assumption on $\calM$.}
Several of our strongest results require $\calM$ to be contractible, and
Section~\ref{sec:commentcontractible} shows that natural weakenings of this
assumption are insufficient.
Under what broader geometric or topological conditions on $\calM$ can similar
results still be obtained?

\item \textbf{Finite regularity.}
Our results assume $f$ is $C^\infty$.
To what extent do the conclusions persist under finite regularity $C^p$?
While $C^1$ regularity is insufficient in general, it is natural to ask whether sufficiently high regularity (for instance $p \ge 2$ or $p \geq 3$) already guarantees the same structural conclusions.

\item \textbf{Quantitative control on $\psi$.}
Theorem~\ref{thm:plnonlinlstsq} constructs a global diffeomorphism $\psi$
which reveals the nonlinear least-squares nature of $f$.
What sort of additional regularity assumptions on $f$ would allow for quantitative control of $\psi$?
For example, if the gradient of $f$ is $L$-Lipschitz continuous around $S$, then for $x \in S$ it is easy to see that $\D\psi(x)$ has $m$ singular values equal to 1 (due to $\D\pi(x)$ being identity on $\T_x S$) and $k$ singular values in the interval $[\sqrt{\mu/2}, \sqrt{L/2}]$ due to $\D\varphi(x)$ and the equality $\nabla^2 f(x) = 2 \, \D\varphi(x)^* \circ \D\varphi(x)$.
How can we assert control on $\psi$ away from $S$?
Such information could have implications for the analysis of optimization methods.


\item \textbf{Global normal forms beyond positive curvature.}
Theorem~\ref{thm:globalsquaredistsinglecppl} can be viewed as a global analogue
of the Morse lemma when the Hessian at the unique critical point is positive
definite.
The classical Morse lemma applies to nondegenerate critical points of arbitrary
signature (e.g., nondegenerate saddle points).
Under what global assumptions can one obtain analogous global descriptions for
functions exhibiting saddle structure?



\item \textbf{Characterization of admissible minimizer sets, including their embedding.}
If $\calM = \Rn$, Corollary~\ref{cor:equivonS} shows that a smooth manifold is diffeomorphic to the
minimizer set of a smooth, globally \pl{} function if and only if it is contractible.
This is an \emph{intrinsic} topological condition.
It does not address how such a manifold may be \emph{embedded}
in $\calM$.
What additional topological conditions on an embedding
$S\subseteq\calM$ ensure---or obstruct---the existence of a smooth, globally \pl{}
function $f \colon \calM \to \reals$ having $S$ as its minimizer set?
We expand briefly on this question in Appendix~\ref{app:knottheory}, in relation to knot theory.
\end{itemize}

\section*{Acknowledgments}

We thank Andreea-Alexandra Mu\c{s}at, Moishe Kohan, Jaap Eldering, Matthew Kvalheim, Colin Guillarmou and Kenneth Falconer for helpful discussions.

\section*{Funding}

This work was supported by the Swiss State Secretariat for Education, Research and Innovation (SERI) under contract number MB22.00027.

\clearpage

\appendix

\section{Morse lemma at a local minimizer} \label{sec:morselemma}

Lemma~\ref{lem:morselemmaposdef} below is the \emph{Morse Lemma} specialized to local minimizers.
For completeness, we include a simple proof that follows the approach of \citet[\S{C.6}]{hormander2007analysis}.
See also~\citep[\S6.1]{hirsch1976differential} or~\citep[Lem.~2.2]{milnor1964difftopochapter} for the statement at general nondegenerate critical points, and see~\citep[Lem.~3.51]{banyaga2004morsehomology} for a Morse--Bott extension.
If $f$ is not $C^\infty$ smooth, the argument below loses two orders of regularity; see~\citep{ostrowski1968morsekuiper} for a proof that loses only one order of regularity.

\begin{lemma}[Morse Lemma at a local minimizer] \label{lem:morselemmaposdef}
    Let $f \colon \calM \to \reals$ be smooth on a Riemannian manifold $\calM$ of dimension $n$ (not necessarily connected or complete).
    Let $\bar x$ be a critical point of $f$ at which the Hessian of $f$ is positive definite.
    Then there exist an $\epsilon > 0$ and a map $\varphi \colon B_\epsilon \to \Rn$ with $B_\epsilon = \{ x \in \calM : \dist(x, \bar x) < \epsilon \}$ such that:
    \begin{enumerate}
        \item $\varphi$ is a diffeomorphism from $B_\epsilon$ to its image with $\varphi(\bar x) = 0$; and 
        \item $f(x) = f(\bar x) + \|\varphi(x)\|^2$ for all $x \in B_\epsilon$.
    \end{enumerate}
    In particular, for all $r > 0$ sufficiently small, $\varphi$ maps the (local) sublevel set $\{ x \in B_\epsilon : f(x) < f(\bar x) + r^2 \}$ diffeomorphically to an open Euclidean ball of radius $r$.
\end{lemma}
\begin{proof}
    Select a normal coordinates chart around $\bar x$, that is, some $\bar \epsilon > 0$ and diffeomorphism $\phi \colon B_{\bar \epsilon} \to V$ with $V = \phi(B_{\bar \epsilon}) \subset \Rn$ such that $t \mapsto \phi^{-1}(t \phi(x))$ is the minimizing geodesic from $\bar x$ (at $t= 0$) to $x$ (at $t = 1$).
    In particular, $\dist(x, \bar x) = \|\phi(x)\|$, so that $\phi(\bar x) = 0$ and also $V = B_{\bar \epsilon}^n := \{ v \in \Rn : \|v\| < \bar\epsilon \}$.

    Passing to those coordinates,
    let $\tilde f = f \circ \phi^{-1} \colon B_{\bar \epsilon}^n \to \reals$.
    Deduce from $\grad f(\bar x) = 0$ and $\Hess f(\bar x) \succ 0$ that $\tilde f(0) = f(\bar x)$, $\grad \tilde f(0) = 0$ and $\Hess \tilde f(0) \succ 0$.
    Fix $v \in B_{\bar \epsilon}^n$ and let $g(t) = \tilde f(tv)$.
    Since $g \colon [0, 1] \to \reals$ is smooth, we know $g(1) = g(0) + g'(0) + \int_0^1 \int_0^t g''(s) \, \ds \dt$ where $g'(t) = v^\top \grad \tilde f(tv)$ and $g''(t) = v^\top \Hess \tilde f(tv)[v]$.
    Thus,
    \begin{align*}
        \tilde f(v) = f(\bar x) + v^\top H(v)[v] && \textrm{ with } && H(v) = \int_0^1 \int_0^t \Hess \tilde f(sv) \, \ds \dt.
    \end{align*}
    In particular, $H(0) = \frac{1}{2} \Hess \tilde f(0)$ is positive definite.
    By continuity of the Hessian of $\tilde f$, there exists $\epsilon \in (0, \bar\epsilon]$ such that $\Hess \tilde f(v) \succ 0$ for all $v \in B_{\epsilon}^n$.
    Thus, $H(v)$ also is positive definite for all $v \in B_\epsilon^n$.
    Taking matrix square roots or via Cholesky decomposition, it follows that there exists a smooth map $R \colon B_\epsilon^n \to \Rnn$ such that
    \begin{align*}
        H(v) = R(v)^\top R(v) && \textrm{ for all } && v \in B_\epsilon^n,
    \end{align*}
    and of course each $R(v)$ is invertible.

    Let $\tilde \varphi(v) = R(v) v$, defined from $B_\epsilon^n$ to $\Rn$.
    Therefore,
    $$\tilde f(v) = f(\bar x) + v^\top H(v)[v] = f(\bar x) + v^\top R(v)^\top R(v)[v] = f(\bar x) + \|\tilde \varphi(v)\|^2.$$
    The differential $\D\tilde\varphi(v)[\dot v] = \D R(v)[\dot v] v + R(v) \dot v$ simplifies at $v = 0$ to $\D\tilde\varphi(0) = R(0)$, which is invertible.
    Thus, by the inverse function theorem, we may reduce $\epsilon > 0$ if need be so that $\tilde\varphi$ is a diffeomorphism from $B_\epsilon^n$ to its image.

    To conclude, we have that $\varphi = \tilde \varphi \circ \phi|_{B_{\epsilon}}$ is indeed a diffeomorphism from $B_\epsilon$ to its image, and that $f(x) = \tilde f(\phi(x)) = f(\bar x) + \|\tilde \varphi(\phi(x))\|^2 = f(\bar x) + \|\varphi(x)\|^2$, as announced.
\end{proof}

\section{Extension of the Morse lemma diffeomorphism}\label{app:extension-morse}


Lemma~\ref{lem:localsquaredistsinglecp} provides a \emph{global} diffeomorphism of $\calM$ with the same \emph{local} effect as the Morse lemma above (Lemma~\ref{lem:morselemmaposdef}).

\begin{proof}[Proof of Lemma~\ref{lem:localsquaredistsinglecp}]
  Let $B^n$ denote the open unit ball in $\Rn$.
  Select two diffeomorphisms from $B^n$ to neighborhoods of $x^*$ in $\calM$, as follows:
  \begin{enumerate}
    \item Using the Morse Lemma (Lemma~\ref{lem:morselemmaposdef} and appropriate rescaling), select $\epsilon, r > 0$ (both less than the injectivity radius of $\calM$ at $x^*$) and a diffeomorphism $\varphi_a^{-1} \colon B^n \to U_a = \{ x \in \calM : \dist(x, x^*) < \epsilon \textrm{ and } f(x) < f(x^*) + r^2 \}$ such that $f(\varphi_a^{-1}(v)) = f(x^*) + r^2 \|v\|^2$; and

    \item Using a normal coordinates chart around $x^*$ and appropriate rescaling, select a diffeomorphism $\varphi_b^{-1} \colon B^n \to U_b = \{ x \in \calM : \dist(x, x^*) < r \}$ such that $\dist(\varphi_b^{-1}(v), x^*) = r\|v\|$ (same $r$ as above). 
  \end{enumerate}

  We aim to apply the Palais--Cerf theorem to extend the diffeomorphism $\varphi_a^{-1} \circ \varphi_b \colon U_b \to U_a$ to a diffeomorphism $\psi \colon \calM \to \calM$ such that $\psi \circ \varphi_b^{-1} = \varphi_a^{-1}$.
  See~\citep[Thm.~B]{palais1960extending}, and also \citep[Lem.~2]{milnor1964difftopochapter}, \citep[Ch.~8, Thm.~3.1]{hirsch1976differential} and \citep{goldstein2025gluingdiffeos}, where the latter handles $C^p$ regularity with $p < \infty$ explicitly.

  If $\calM$ is not orientable, then that theorem applies directly.
  If $\calM$ is orientable, then we need to ensure that the diffeomorphism $\varphi_a^{-1} \circ \varphi_b$ preserves orientation,\footnote{If $\calM$ is oriented, then the sign of the determinant of the differential of $\psi$ is well defined throughout $\calM$, and it must be positive because the diffeomorphism $\psi$ produced by the Palais--Cerf theorem is identity outside a compact set, and the determinant cannot be zero anywhere for a diffeomorphism so it cannot change sign. If $\calM$ is not orientable, there is no such obstruction.}
  that is, we must check that the determinant of the differential of $\varphi_a^{-1} \circ \varphi_b$ is positive at $x^*$ (this determinant is well defined because $(\varphi_a^{-1} \circ \varphi_b)(x^*) = x^*$ so we can express the differential as an $n \times n$ matrix with respect to an arbitrary basis of $\T_{x^*}\calM$).
  If it is not, then simply redefine $\varphi_b$ by flipping the sign of one of the coordinates: this does not change the properties we had required for $\varphi_b$, and now the Palais--Cerf theorem applies.


  In all cases, Palais--Cerf provides a diffeomorphism $\psi \colon \calM \to \calM$ such that $\psi \circ \varphi_b^{-1} = \varphi_a^{-1}$.
  For $x$ such that $\dist(x, x^*) < r$, it follows from the properties of $\psi$, $\varphi_a$ and $\varphi_b$ that
  \begin{align*}
    (f \circ \psi)(x) = f(\varphi_a^{-1}(\varphi_b(x))) = f(x^*) + r^2 \|\varphi_b(x)\|^2 = f(x^*) + \dist(x, x^*)^2.
  \end{align*}
  By continuity, this extends to the non-strict inequality $\dist(x, x^*) \leq r$.
\end{proof}


\section{Rescaled gradient flow}\label{app:rescaledflow}

Lemma~\ref{lem:GFrescaledoncoerciveinvexM} provides simple statements about normalized gradient flow.
Versions of this lemma appear in various places (e.g., in passing in~\citep[Thm.~3.1]{milnor1963morse}).
We include the details here so we have a specific version we can rely on.

\begin{proof}[Proof of Lemma~\ref{lem:GFrescaledoncoerciveinvexM}]
    Fix $x \neq x^*$.
    By design, for $t$ in the domain of definition of $t \mapsto \nu(x, t)$ we have
    \begin{align*}
        \ddt f(\nu(x, t)) = \innerBig{\grad f(\nu(x, t))}{\ddt \nu(x, t)}_{\nu(x, t)} = 1.
    \end{align*}
    Thus, if the trajectory is defined from time 0 up to (or down to) $t$, then $f(\nu(x, t)) = f(\nu(x, 0)) + t = f(x) + t$.
    Note from Lemma~\ref{lem:coerciveuniquecpglobalmin} that $f$ is nonnegative.

    To see that the flow is defined on the stated interval, pick arbitrary function values $0 < a < b < \infty$ and a corresponding smooth bump function $\beta \colon \calM \to \reals$ such that $\beta(x) = 1$ if $a \leq f(x) \leq b$, and $\beta(x) = 0$ if $f(x) \leq a/2$ or $f(x) \geq 2b$~\citep[Prop.~2.25]{lee2012smoothmanifolds}.
    Recall the definition $W(x) = \frac{1}{\|\grad f(x)\|_x^2} \grad f(x)$.
    Let $\hat W(x) = \beta(x) W(x)$, understood to be identically zero where $\beta$ is so.
    In particular, $\hat W$ is smooth because it is zero in a neighborhood of the only critical point of $f$ (which is where $W$ loses smoothness). 
    The flows on $W$ (our target) and on $\hat W$ coincide in the region $\{ x : a \leq f(x) \leq b \}$.
    Moreover, $\hat W$ is
    compactly supported (because the sublevel sets of $f$ are compact),
    hence its trajectories are smoothly defined for all times~\citep[Thm.~9.16]{lee2012smoothmanifolds}.
    Together with the preliminary observation above and the fact that $(a, b)$ can be taken arbitrarily close to $(0, \infty)$, we find that $t \mapsto \nu(x, t)$ is defined for all $t$ such that $f(\nu(x, t))$ is in $(0, \infty)$, that is, for all $t \in (-f(x), \infty)$.
    The map $\nu$ is smooth on this domain by the fundamental theorem of flows~\citep[Thm.~9.12]{lee2012smoothmanifolds}

    A trajectory accumulates only at points where $f = 0$ because $\lim_{t \to -f(x)} f(\nu(x, t)) = 0$; 
    but those are global minimizers hence critical, and $x^*$ is the only critical point.
    Thus, for $t \to -f(x)$, the trajectory stays in a compact set and its single accumulation point is the origin.
    Therefore, it converges to that point.
\end{proof}

\section{Global transporter of tangent vectors} \label{app:transporter}

In the proof of Theorem~\ref{thm:pitrivial}, we use the following technical fact from differential geometry, applied to $\calN = S$.
The map $T$ is sometimes called a \emph{transporter}; the construction below matches~\citep[\S10.5, Prop.~10.66]{boumal2020intromanifolds}.


\begin{lemma} \label{lem:transporter}
  Let $\calN$ be a smooth manifold.
  There exists a smooth map
  \begin{align*}
    T \colon \T\calN \times \calN \to \T\calN \colon ((x, v), y) \mapsto T_{y \leftarrow x}(v)
  \end{align*}
  with the following properties:
  \begin{enumerate}
    \item $v \mapsto T_{y \leftarrow x}(v)$ is a linear map from $\T_{x} \calN$ to $\T_{y} \calN$ for all $x, y \in \calN$, and
    \item $T_{x \leftarrow x}(v) = v$ for all $(x, v) \in \T \calN$.
  \end{enumerate}
  In particular, $V(y) := T_{y \leftarrow x}(v)$ defines a smooth vector field on $\calN$ such that $V(x) = v$.
\end{lemma}
\begin{proof}
  There are many ways to build such a map.
  A brief argument goes as follows:
  \begin{enumerate}
    \item If not already the case, embed $\calN$ into a Euclidean space $\calE$ (say, $\Rd$ with $d = 2\dim\calN+1$ using Whitney's embedding theorem~\citep[Thm.~6.15]{lee2012smoothmanifolds}).

    \item Let $\Proj_x$ be the orthogonal projector (with respect to the Euclidean metric) from $\calE$ to $\T_x \calN$ (as a linear subspace of $\calE$): this depends smoothly on $x$.
    
    Indeed, for each $\bar{x} \in \calN$ we can choose a neighborhood $U$ of $\bar{x}$ in $\calE$ and a smooth local defining function $h \colon U \to \reals^{\dim\calE - \dim\calN}$ such that $\calN \cap U = h^{-1}(0)$ and $\D h(x)$ is surjective for all $x \in U$. Then, $\T_x\calN = \ker \D h(x)$ for all $x \in \calN \cap U$ and therefore $\Proj_x = I_{\calE} - \D h(x)^\dagger \circ \D h(x) = I_{\calE} - \D h(x)^* \circ (\D h(x) \circ \D h(x)^*)^{-1}  \circ \D h(x)$.

    \item Define $T_{y \leftarrow x}(v) = \Proj_y(v)$, where $v \in \T_x\calN$ is seen as a vector in $\calE$.
    \qedhere
  \end{enumerate}
\end{proof}

\section{Stabilizing $S$ by $\reals^k$ yields $\reals^n$}\label{app:deeptheorems}

This appendix outlines a proof of Theorem~\ref{thm:deepthm2}.
It relies on the following known result, which follows from a long line of classical works.

\begin{theorem}[\citet{stallings1962piecewise,HuschPrice1970,Perelman20022003}]\label{thm:deepthm1}
Let $S$ be a (non-empty) contractible smooth manifold.
\begin{itemize}
\item[(a)] If $\dim(S) \leq 2$, then $S$ is diffeomorphic to $\reals^{\dim(S)}$.
\item[(b)] If $\dim(S) = 3$ or $\dim(S) \geq 5$, and $S$ is simply connected at infinity (see Remark~\ref{rem:simplyconnectedatinfinity}), then $S$ is diffeomorphic to $\reals^{\dim(S)}$.
\end{itemize}
\end{theorem}

\begin{proof}
For item~(a): if $\dim(S) = 0$, then $S$ is a singleton; if $\dim(S)=1$, see~\citep[Thm.~5.27]{lee2011topological}; and if $\dim(S)=2$, see~\citep[Ex.~1B.2]{hatcher2002algebraictopo}.  
For item~(b):
\begin{itemize}
\item If $\dim(S) \geq 5$, the result follows immediately from~\citet[Thm.~5.1]{stallings1962piecewise}.  
\item If $\dim(S) = 3$, then $S$ is homeomorphic to $\reals^3$ following~\citet{HuschPrice1970} together with Perelman's proof of the Poincaré conjecture~\citep{Perelman20022003}; see also~\citep[Thm.~3.5.3]{Guilbaultsurvey2016}.  
By Moise's theorem~\citep{Moisestheorem1952}, this homeomorphism can be promoted to a diffeomorphism.  \qedhere
\end{itemize}
\end{proof}

Note that item~(b) fails in dimension four, due to the existence of exotic $\reals^4$~\citep[Thm.~8.4C]{FreedmanQuinn1990}.  
For the topological (homeomorphism) statement of Theorem~\ref{thm:deepthm1}, see~\citep{Freedman4manifolds1982,Guilbault1992}.

\begin{proof}[Proof of Theorem~\ref{thm:deepthm2}]
For notational convenience, we write $S$ in place of $\tilde S$ throughout this proof.
The ``only if'' direction is trivial: if $S \times \reals^k$ is diffeomorphic to a linear space, then it is contractible; and it is homotopy equivalent to $S$, so $S$ is contractible as well.

The ``if'' direction is the culmination of results by many authors, and can be split into several cases.
\begin{itemize}
\item $\dim(S) \leq 2$: follows immediately from Theorem~\ref{thm:deepthm1}(a).  

\item $\dim(S) \geq 4$: an immediate consequence of~\citep[Cor.~5.3]{stallings1962piecewise}, itself derived from Theorem~\ref{thm:deepthm1}(b) (using $\dim(S \times \Rk) = \dim(S) + k \geq 5$).

\item $\dim(S) = 3$: \citet[Thm.~5]{luft1987contractible}, building on~\citep{McMillan1961}, showed that $S \times \reals^k$ is piecewise-linearly homeomorphic to $\reals^{3+k}$ provided there are no fake $3$-cells.  
The nonexistence of fake $3$-cells follows from Perelman's proof of the Poincaré conjecture~\citep{Perelman20022003}.  
Finally, by~\citep[Cor.~6.6]{MunkresObstructionstoSmoothing1960}, the piecewise-linear structure can be smoothed, yielding a diffeomorphism.   \qedhere 
\end{itemize}
\end{proof}

\begin{remark}[Simply connected at infinity]\label{rem:simplyconnectedatinfinity}
The assumption of simple connectivity at infinity in Theorem~\ref{thm:deepthm1}(b) is essential: the Whitehead manifold is a contractible $3$-manifold not homeomorphic to $\reals^3$ precisely because it is not simply connected at infinity.  

Formally, a space $X$ is \emph{simply connected at infinity}~\citep{stallings1962piecewise} if, for every compact set $C \subseteq X$, there exists a compact $D \supseteq C$ such that every loop in $X \setminus D$ can be contracted to a point within $X \setminus C$.  
For example:
\begin{itemize}
\item $\reals^n$ ($n \geq 3$), or $\reals^n$ with finitely many points removed, is simply connected at infinity.  
\item $\reals^2$, the cylinder, and the Whitehead manifold are \emph{not} simply connected at infinity.  
\end{itemize}
\end{remark}

\section{Contractible manifolds that are compact are singletons}\label{app:compactimpliespoint}

Let $S$ be a non-empty, compact and contractible smooth manifold without boundary.
What are such spaces?
A point is certainly one example.
What are other examples?
A closed ball is \emph{not} an example since it has a boundary.
A sphere is also \emph{not} an example since it is not contractible.
It turns out single points are the \emph{only} examples.
This is a well-known fact (see for example~\citep[Ex.~2.4.6, p.~83]{guillemin1974difftop}).
We sketch a proof here.

\begin{proposition}\label{prop:compactimpliessingleton}
Let $S$ be a compact and contractible topological manifold (non-empty, without boundary).
Then $S$ is a point.
\end{proposition}
\begin{proof}
Let $n$ denote the dimension of $S$.
Since $S$ is contractible, it is simply connected, and so orientable~\citep[Prop.~3.25]{hatcher2002algebraictopo} (see also Section~\ref{sec:commentcontractible}).
As $S$ is also ``closed'' (compact without boundary), the top homology group of $S$ is $H_n(S) = \mathbb{Z}$~\citep[Thm.~3.26]{hatcher2002algebraictopo}.
Owing to contractibility again, $S$ has the same homology groups as a point, because homology groups are invariant under homotopy equivalence~\citep[\S2.1]{hatcher2002algebraictopo}.
The homology groups of a point are $H_0 = \mathbb{Z}$ and $H_k = 0$ for $k \geq 1$.
If $n \geq 1$, we conclude that $\mathbb{Z} = H_n(S) = 0$: a contradiction.
Thus, $n=0$ and hence $S$ is a collection of points.
Since $S$ is connected (by contractibility), $S$ must be a single point.
\end{proof}

\begin{proof}[Proof of Corollary~\ref{cor:compactSsingleton}]
By Lemma~\ref{lem:MB} and Proposition~\ref{prop:picontinuous}, $S$ is a contractible smooth manifold because $\calM$ is contractible.
If $S$ is compact, Proposition~\ref{prop:compactimpliessingleton} implies $S$ is a singleton.
\end{proof}

\section{Remarks related to knot theory}\label{app:knottheory}

One of the open questions listed in Section~\ref{sec:conclusions} asks: what may $S$ look like as an \emph{embedded} submanifold of $\calM$?
(This is different from asking what $S$ is diffeomorphic to, as answered by Corollary~\ref{cor:equivonS}.)

Let us expand on this question.
Assume that $\calM$ is contractible.
Theorems~\ref{thm:plnonlinlstsq} and~\ref{thm:constructingglobalPLcor}
together imply that a properly embedded submanifold
$S \subseteq \calM$ arises as the minimizer set of a globally \pl{}
function if and only if there exists a diffeomorphism
\begin{align}\label{eq:tobesatisfiedbypsi}
\psi \colon \calM \to S \times \reals^k
\quad \text{with} \quad
\psi(S)=S\times\{0\}.
\end{align}
What topological conditions on the embedding
$S \subseteq \calM$ guarantee---or rule out---the existence of such a
diffeomorphism $\psi$?

A first necessary condition follows from the topology of the complement.
If a diffeomorphism $\psi$ satisfying~\eqref{eq:tobesatisfiedbypsi} exists, then
\[ \calM \setminus S \;\cong\; (S\times\reals^k)\setminus(S\times\{0\}) \;\cong\; S\times(\reals^k\setminus\{0\}). \]
Assume the codimension of $S$ satisfies $k = \codim(S) \geq 2$.
Then, $\calM \setminus S$ is path-connected and the fundamental groups ($\pi_1$) obey
\[ \pi_1(\calM\setminus S) \;\cong\; \pi_1(S)\times\pi_1(\reals^k\setminus\{0\}) \;\cong\; \pi_1(S)\times\pi_1(\mathbb{S}^{k - 1}), \]
where $\mathbb{S}^{k - 1}$ denotes the sphere of dimension $k-1$.
Since $S$ must itself be contractible, this yields the necessary condition
\begin{equation}\label{eq:complementcondition}
\pi_1(\calM\setminus S)\;\cong\;\pi_1(\mathbb{S}^{k - 1}),
\end{equation}
which is trivial for $k\ge 3$ and isomorphic to $\mathbb Z$ for $k=2$.

Consequently, any embedding $S \subseteq \calM$ violating~\eqref{eq:complementcondition}
cannot arise as the minimizer set of a globally \pl{} function.
For example, a nontrivial long knot\footnote{A \emph{long knot} is a proper smooth embedding
$\reals \hookrightarrow \reals^3$ that agrees with a fixed linear
embedding outside a compact set~\citep{budneylongknots2007}.}
in $\reals^3$ has complement with
nonabelian fundamental group~\citep[Ch.~3, 4]{rolfsen1976knots}:
it cannot occur as such a minimizer set, since
\eqref{eq:complementcondition} would force the fundamental group of the complement to be $\mathbb{Z}$, which is
abelian---see also~\citep[Ex.~9 in \S8.1, p.~183]{hirsch1976differential}.

It is natural to ask whether the complement condition
\eqref{eq:complementcondition} is also sufficient for the existence of
a diffeomorphism $\psi$ satisfying~\eqref{eq:tobesatisfiedbypsi}.
We suspect that this is not the case in full generality, and leave a
precise characterization of admissible embeddings as an open problem.

A related question concerns the role of codimension.
If $S$ is contractible and its codimension $k$ is sufficiently large,
does a diffeomorphism $\psi$ satisfying~\eqref{eq:tobesatisfiedbypsi} always exist?
The example of long knots provides useful intuition:
while embeddings $\reals \hookrightarrow \reals^3$ may be knotted,
embeddings of $\reals \hookrightarrow \reals^4$ can be untangled up to ambient isotopy.
This suggests that low-codimension obstructions may disappear in higher
codimension in this context too.



\bibliographystyle{abbrvnat}
\bibliography{boumal,bibliography}

\end{document}